\documentclass[12pt]{article}

\usepackage{amsmath, amsthm, amssymb}
\usepackage{mathptmx}

\usepackage{graphicx}
\usepackage{float}
\usepackage{multirow}
\usepackage{tikz}
\usepackage{enumerate}

\usepackage{algorithm}
\usepackage{algpseudocode} 

\usepackage{subcaption}

\usepackage[a4paper]{geometry}
\usepackage{xcolor}
\usepackage{hyperref}

\newtheorem{theorem}{Theorem}[section]
\newtheorem{proposition}{Proposition}[section]
\newtheorem{remark}{Remark}[section]
\newtheorem{lemma}{Lemma}[section]

\newtheorem{assumption}{Assumption}

\begin{document}

\title{Multi-period Value-at-Risk Constrained Portfolio Optimization via DC Programming}

\author{
Nguyen Thi Thu Van%
\thanks{University of Economics Ho Chi Minh City, Vietnam.
Email: \texttt{van.nguyen@ueh.edu.vn}}
}

\date{\today}

\maketitle
\begin{abstract}

In this paper, we study a multi-period portfolio optimization problem with
finite-scenario Value-at-Risk (VaR) constraints. The objective balances
expected return against proportional transaction costs and a quadratic
regularizer that discourages portfolio concentration. Using a
finite-scenario VaR--CVaR identity, we derive a penalized
difference-of-convex (DC) formulation over the underlying convex
portfolio set. We prove the existence of optimal solutions and give a
sufficient condition under which the original constrained problem and
its penalized formulation have the same global minimizers. To solve the
penalized DC problem, we propose a projected inertial boosted
difference-of-convex algorithm (iBDCA) that combines inertial
extrapolation, an objective safeguard, and a boosted line search. We
prove that the algorithm is well defined, generates a nonincreasing
sequence of penalized objective values, and that every accumulation
point is critical. Under a local no-ties condition at an accumulation
point, we further establish convergence of the whole sequence to that
point with a local \(R\)-linear rate. Numerical experiments compare
iBDCA with DCA and BDCA on matched problem instances and examine the
robustness of the methods with respect to the VaR threshold.
Out-of-sample backtests additionally include equal-weight and
buy-and-hold benchmarks to assess realized performance, downside risk,
and empirical VaR control.

\end{abstract}


\section{Introduction}

Portfolio optimization under risk constraints is a fundamental problem
in mathematical finance and operations research. Since the seminal
work of Markowitz \cite{Markowitz1952}, many models have been
developed to balance expected return and financial risk. Among the risk
measures used to quantify potential portfolio losses, Value-at-Risk
(VaR) remains an important benchmark because of its intuitive
quantile-based interpretation \cite{BasakShapiro2001, Hull2023,Ruppert2011}. Its
relevance is further supported by empirical properties of financial
returns, including heavy tails, volatility clustering, and departures
from Gaussian behavior
\cite{Cont2001,Embrechts1997,Mandelbrot1963}. However, incorporating
VaR into portfolio optimization creates computational difficulties
because the resulting problems are generally nonsmooth and nonconvex.

For a discrete return distribution represented by finitely many
scenarios, portfolio VaR is continuous and piecewise affine and can be
expressed as the difference of two Conditional Value-at-Risk (CVaR)
functions. Since these CVaR functions are convex and piecewise affine
\cite{Krokhmal2002,RockafellarUryasev2000,
RockafellarUryasev2002}, this representation makes DC optimization
methods applicable. DC programming and DCA approaches for VaR and
VaR-constrained portfolio problems were studied in
\cite{PhamDinhNguyenLeThi2009,Wozabal2012,
WozabalHochreiterPflug2010}. Related DC portfolio models involving
downside risk measures and cardinality constraints were considered in
\cite{LeThiMoeiniPhamDinh2009a}. Multi-period portfolio models have
also been investigated in mean--variance and stochastic programming
frameworks \cite{LiNg2000,MulveyVladimirou1992}. A multi-period model
involving VaR, transaction costs, and practical constraints was
considered in \cite{BabazadehEsfahanipour2019}, while transaction
costs in portfolio optimization have been studied in
\cite{Constantinides1986,LeThiMoeiniPhamDinh2009b}. More recently,
Thormann et al. developed a BDCA framework for a single-period
portfolio problem with a VaR constraint
\cite[Algorithm~2]{ThormannVuongZemkoho2026}. However, existing DC
approaches to VaR-constrained portfolio optimization are primarily
single-period, while a unified multi-period DC treatment of periodwise
VaR constraints, intertemporal transaction costs, and diversification
regularization remains limited.

Inspired by this framework, we consider a multi-period VaR-constrained
portfolio model with proportional transaction costs and
diversification-promoting quadratic regularization. At each planning
period, portfolio allocations satisfy budget, nonnegativity, and VaR
constraints. Transaction costs couple consecutive allocations, while
the quadratic regularizer discourages excessive portfolio
concentration. Using a finite-scenario VaR--CVaR identity, we formulate
the model as a penalized DC program over the underlying convex
portfolio set. We establish the existence of optimal solutions and give
a sufficient condition under which the original constrained problem
and its penalized formulation have the same global minimizers.

To solve the penalized problem, we propose a projected inertial BDCA
(iBDCA). The method first projects an inertial extrapolation onto the
feasible set and applies an objective safeguard. It then solves a
strongly convex DCA subproblem and performs a boosted line search. The
inertial step generates extrapolated candidate points, while the
objective safeguard preserves descent of the penalized objective. The
method also builds on the general DCA and BDCA frameworks developed in
\cite{AragonArtachoFlemingVuong18,AragonArtachoVuong20,
LeThiPhamDinh2005}. We prove that the proposed algorithm is well defined, that it generates
a nonincreasing sequence of penalized objective values, and that every
accumulation point is a critical point of the penalized DC problem. The
extended-valued objective associated with the multi-period problem is
piecewise linear--quadratic and satisfies the
Kurdyka--\L{}ojasiewicz property with exponent \(1/2\). Under a local
no-ties condition at an accumulation point, we further establish
convergence of the whole sequence to that point with a local
\(R\)-linear rate
\cite{AttouchBolteSvaiter2013,Bolte2007,LiPong2018}.

Numerical experiments on two multi-asset datasets are conducted in
three parts. First, matched solver experiments compare iBDCA with DCA
and BDCA on identical problem instances in terms of attained penalized
objective values, scenario-based VaR feasibility, iteration counts,
and solution times. Second, robustness experiments examine the
performance of the three methods under changes in the VaR threshold.
Finally, out-of-sample backtests compare the resulting portfolio
strategies with equal-weight and buy-and-hold benchmarks using final
wealth, the Sharpe ratio, maximum drawdown, empirical exceedance rates,
and VaR residuals.

The remainder of the paper is organized as follows.
Section~\ref{sec2} reviews the preliminary concepts.
Section~\ref{sec3} introduces the portfolio model, establishes the
existence of optimal solutions, develops its penalized DC
reformulation, and gives a sufficient condition for global equivalence
between the original and penalized problems. Section~\ref{sec4}
presents the projected inertial BDCA and analyzes its convergence
properties. Section~\ref{sec5} reports the numerical experiments, and
Section~\ref{sec6} concludes the paper.

\section{Preliminaries}
\label{sec2}

In this section, we recall the basic concepts from DC programming and
risk measurement that will be used throughout the paper; see, e.g.,
\cite{BoydVandenberghe2004,TaoAn97,
RockafellarWets1998}. We also state an elementary sequence lemma used
later in the convergence analysis.

For \(u,v\in\mathbb R^d\), we denote by
\(\langle u,v\rangle\) the standard Euclidean inner product and by
\(\|\cdot\|\) the corresponding Euclidean norm, where \(d\in \mathbb N=\{1,2, \dots \} \). For
\(u\in\mathbb R^d\), the \(\ell_1\)-norm is defined by
\( 
\|u\|_1
:=
\sum_{i=1}^{d}|u_i|.
\) 

Throughout the paper, the multi-period decision variable is written in
block form as
\( 
x=(x^1,\ldots,x^T)\in(\mathbb R^n)^T,
\)
where \(x^t\in\mathbb R^n\) denotes the portfolio allocation at period
\(t\). We identify \((\mathbb R^n)^T\) with \(\mathbb R^{nT}\). Under
this identification, the Euclidean norm satisfies
\( 
\|x\|^2
=
\sum_{t=1}^{T}\|x^t\|^2.
\) 

We first recall some basic concepts from DC programming. A finite-valued
function \(\Phi:\mathbb R^{nT}\to\mathbb R\) is called a
\emph{difference-of-convex (DC) function} if it admits a representation
\begin{equation}
\label{eq:generic_dc}
\Phi(x)=G(x)-H(x),
\; x\in\mathbb R^{nT},
\end{equation}
where \(G,H:\mathbb R^{nT}\to\mathbb R\) are convex functions. The pair
\((G,H)\) is called a DC decomposition of \(\Phi\), which is generally
not unique; see
\cite{LeThiPhamDinh2005,TaoAn97}.

For a proper convex function
\(f:\mathbb R^{nT}\to(-\infty,+\infty]\) and
\(x\in\operatorname{dom}f\), its convex subdifferential is defined by
\[
\partial f(x)
=
\left\{
v\in\mathbb R^{nT}:
f(y)\geq f(x)+\langle v,y-x\rangle,
\; \forall y\in\mathbb R^{nT}
\right\}.
\]
If \(f\) is differentiable at \(x\), then
\( 
\partial f(x)=\{\nabla f(x)\}.
\) 

Let \(C\subset\mathbb R^{nT}\) be a nonempty closed convex set. The
normal cone to \(C\) at \(x\in C\) is defined by
\[
N_C(x)
=
\left\{
v\in\mathbb R^{nT}:
\langle v,y-x\rangle\leq0,
\; \forall y\in C
\right\}.
\]

A point \(\bar x\in C\) is called a \emph{critical point} of the
constrained DC problem
\[
\min_{x\in C}\;G(x)-H(x)
\]
if
\( 
\partial H(\bar x)
\cap
\bigl(
\partial G(\bar x)+N_C(\bar x)
\bigr)
\neq\varnothing.
\) 
Every local minimizer of the constrained DC problem is a critical
point; see
\cite{LeThiPhamDinh2005,TaoAn97}.

\medskip
We next recall the risk measures used in the portfolio model. Let \(L\)
be a real-valued loss random variable. For a confidence level
\(\alpha\in(0,1)\), the Value-at-Risk of \(L\) at level \(\alpha\) is
defined by
\[
\operatorname{VaR}_{\alpha}(L)
:=
\inf
\left\{
u\in\mathbb R:
\mathbb P(L\leq u)\geq\alpha
\right\}.
\]
Equivalently,
\( 
\operatorname{VaR}_{\alpha}(L)
=
F_L^{-1}(\alpha),
\) 
where
\[
F_L^{-1}(\alpha)
:=
\inf
\left\{
u\in\mathbb R:
F_L(u)\geq\alpha
\right\}
\]
is the generalized inverse of the cumulative distribution function
\cite{Acerbi2002,RockafellarUryasev2002}.

Let \(z\in\mathbb R^n\) denote a
generic portfolio allocation at a fixed period, and let
\(\xi\in\mathbb R^n\) be the corresponding random asset-return vector.
We adopt the loss convention
\( 
L(z):=-\xi^\top z.
\) 
Thus, a positive value of \(L(z)\) represents a loss, whereas a
negative value represents a gain.

A VaR constraint takes the form
\[
\operatorname{VaR}_{\alpha}\bigl(L(z)\bigr)
=
\operatorname{VaR}_{\alpha}(-\xi^\top z)
\leq\tau.
\]
By the definition of the generalized inverse, this constraint is
equivalent to
\[
\mathbb P(-\xi^\top z\leq\tau)
\geq\alpha.
\]
Hence, with probability at least \(\alpha\), the portfolio loss does
not exceed the prescribed threshold \(\tau\).

The mapping
\(
z\longmapsto
\operatorname{VaR}_{\alpha}(-\xi^\top z)
\) 
is generally nonconvex. Consequently, portfolio optimization problems
with VaR constraints are generally nonconvex and computationally
challenging
\cite{Wozabal2012,WozabalHochreiterPflug2010}.

We next recall Conditional Value-at-Risk, also known as Expected
Shortfall. Suppose that \(L\) is integrable. For
\(\alpha\in(0,1)\), its upper-tail CVaR is given by the
Rockafellar--Uryasev representation
\cite{RockafellarUryasev2000,RockafellarUryasev2002}:
\[
\operatorname{CVaR}_{\alpha}(L)
=
\min_{u\in\mathbb R}
\left\{
u+
\frac{1}{1-\alpha}
\mathbb E\bigl[(L-u)_+\bigr]
\right\},
\]
where
\( 
(a)_+:=\max\{a,0\}.
\)
If \(L\) has an atomless distribution, this representation reduces to
the conditional tail expectation
\[
\operatorname{CVaR}_{\alpha}(L)
=
\mathbb E
\left[
L
\,\middle|\,
L\geq\operatorname{VaR}_{\alpha}(L)
\right].
\]

CVaR is a coherent risk measure
\cite{Acerbi2002,Artzner1999}. Since the portfolio loss is affine in
\(z\), the mapping
\(
z\longmapsto
\operatorname{CVaR}_{\alpha}(-\xi^\top z)
\) 
is convex.

Throughout the paper, VaR and CVaR are applied to losses and therefore
refer to the upper tail of the loss distribution. This differs from
the lower-tail return convention used in
\cite{ThormannVuongZemkoho2026,Wozabal2012,
WozabalHochreiterPflug2010}. Accordingly, the confidence levels and
coefficients in the DC representation below are written under the
upper-tail loss convention.

We now specialize to the discrete return distributions used in the
portfolio model. The following lemma states the uniform VaR--CVaR
identity required for the subsequent DC reformulation.

\begin{lemma}
\label{lem:VaRDC}

Let \(\alpha\in(0,1)\), and let \(\xi\) be a discrete random
asset-return vector taking the values
\(\xi_j\in\mathbb R^n\), \(j=1,\ldots,S\), with probabilities
\(p_j>0\), \(j=1,\ldots,S\), satisfying
\( 
\sum_{j=1}^{S}p_j=1.
\) 

For \(z\in\mathbb R^n\), define the discrete portfolio loss
\(L(z):=-\xi^\top z\), whose scenario values are
\[
L_j(z):=-\xi_j^\top z,
\; j=1,\ldots,S.
\]
Define
\begin{equation}
\label{v_5aug_1}
\bar\varepsilon
:=
\alpha-
\max_{\substack{
J\subseteq\{1,\ldots,S\}\\
\sum_{j\in J}p_j<\alpha
}}
\sum_{j\in J}p_j.
\end{equation}
Then \(0<\bar\varepsilon\leq\alpha\). Moreover, for every
\(0<\gamma<\bar\varepsilon\) and every \(z\in\mathbb R^n\),
\[
\operatorname{VaR}_{\alpha}(L(z))
=
\frac{1-\alpha+\gamma}{\gamma}
\operatorname{CVaR}_{\alpha-\gamma}(L(z))
-
\frac{1-\alpha}{\gamma}
\operatorname{CVaR}_{\alpha}(L(z)).
\]
Consequently, the mapping
\(
z\longmapsto
\operatorname{VaR}_{\alpha}(-\xi^\top z)
\)
is a continuous piecewise-affine DC function.

\end{lemma}

\begin{proof}

The family
\(
\left\{
J\subseteq\{1,\ldots,S\}:
\sum_{j\in J}p_j<\alpha
\right\}
\)
is nonempty because it contains \(J=\varnothing\). Since
\(\{1,\ldots,S\}\) has finitely many subsets, the maximum
in~\eqref{v_5aug_1} exists and belongs to \([0,\alpha)\). Therefore,
\(0<\bar\varepsilon\leq\alpha\).

Fix \(z\in\mathbb R^n\), and set
\(
v:=\operatorname{VaR}_{\alpha}(L(z)).
\)
By the definition of VaR,
\[
\mathbb P(L(z)<v)
<
\alpha
\leq
\mathbb P(L(z)\leq v).
\]
Moreover,
\(
\mathbb P(L(z)<v)
=
\sum_{\{j:L_j(z)<v\}}p_j
\)
is one of the subset sums considered in~\eqref{v_5aug_1}. Hence,
\[
\mathbb P(L(z)<v)
\leq
\alpha-\bar\varepsilon.
\]
For \(0<\gamma<\bar\varepsilon\) and
\(s\in(\alpha-\gamma,\alpha]\), we therefore have
\[
\mathbb P(L(z)<v)
\leq
\alpha-\bar\varepsilon
<
\alpha-\gamma
<
s
\leq
\alpha
\leq
\mathbb P(L(z)\leq v).
\]
Thus, no value smaller than \(v\) has cumulative probability at least
\(s\), whereas \(v\) does. It follows that
\begin{equation}
\label{v_5aug_2}
\operatorname{VaR}_{s}(L(z))
=
v
=
\operatorname{VaR}_{\alpha}(L(z)),
\;
s\in(\alpha-\gamma,\alpha].
\end{equation}

Since \(L(z)\) has finite support, it is integrable. Moreover,
\(0<\alpha-\gamma<\alpha<1\). Hence, for every
\(\beta\in(0,1)\), the integrated-quantile identity
\cite{RockafellarUryasev2002} gives
\[
(1-\beta)\operatorname{CVaR}_{\beta}(L(z))
=
\int_{\beta}^{1}\operatorname{VaR}_{s}(L(z))\,ds.
\]
Applying this identity with \(\beta=\alpha-\gamma\) and
\(\beta=\alpha\), subtracting the resulting equalities, and
using~\eqref{v_5aug_2}, we obtain
\[
\begin{aligned}
(1-\alpha+\gamma)
\operatorname{CVaR}_{\alpha-\gamma}(L(z))
-
&(1-\alpha)
\operatorname{CVaR}_{\alpha}(L(z))
\\
&=
\int_{\alpha-\gamma}^{\alpha}
\operatorname{VaR}_{s}(L(z))\,ds
=
\int_{\alpha-\gamma}^{\alpha}
\operatorname{VaR}_{\alpha}(L(z))\,ds
=
\gamma\operatorname{VaR}_{\alpha}(L(z)).
\end{aligned}
\]
Dividing by \(\gamma>0\) yields
\begin{equation}
\label{v_5aug_3}
\operatorname{VaR}_{\alpha}(L(z))
=
\frac{1-\alpha+\gamma}{\gamma}
\operatorname{CVaR}_{\alpha-\gamma}(L(z))
-
\frac{1-\alpha}{\gamma}
\operatorname{CVaR}_{\alpha}(L(z)).
\end{equation}

Furthermore, for every
\(\beta\in(0,1)\), the CVaR of \(L(z)\) is the optimal value of the
linear program
\[
\begin{aligned}
\operatorname{CVaR}_{\beta}(L(z))
=
\min_{\substack{u\in\mathbb R\\r\in\mathbb R^S}}\;&
u+\frac{1}{1-\beta}\sum_{j=1}^{S}p_jr_j
\\
\text{\rm s.t.}\;&
r_j\geq L_j(z)-u,
\; j=1,\ldots,S,
\\
&
r_j\geq0,
\; j=1,\ldots,S.
\end{aligned}
\]
Since \(L_j(z)=-\xi_j^\top z\) is affine in \(z\), \(\operatorname{CVaR}_{\beta}(L(z))\) is finite-valued, convex, and piecewise affine, and hence continuous.
Therefore,~\eqref{v_5aug_3} expresses
\(
z\longmapsto
\operatorname{VaR}_{\alpha}(-\xi^\top z)
\)
as the difference of two finite-valued convex continuous
piecewise-affine functions. Hence, it is a continuous
piecewise-affine DC function.

\end{proof}
We conclude this section with an elementary lemma used in the
convergence analysis. Its proof is omitted because the result is
standard.

\begin{lemma}
\label{lem:sequence}

Let \(\{a_k\}\) and \(\{b_k\}\) be nonnegative sequences satisfying
\[
a_{k+1}
\leq
\theta a_k+b_k,
\; k\geq0,
\]
where \(0\leq\theta<1\) and \(b_k\to0\). Then
\(
a_k\to0.
\)

\end{lemma}

\section{Multi-Period VaR-Constrained Portfolio Problem and DC Reformulation}
\label{sec3}

In this section, we proceed in two main steps. First, we introduce the
proposed multi-period VaR-constrained portfolio model. Then, we derive
a penalty DC reformulation of the proposed model.

\subsection{Multi-Period Portfolio Model}
\label{subsec:portfolio_problem}

Let \(n\) denote the number of risky assets and \(T\) the number of
planning periods. At each period \(t=1,\ldots,T\), the portfolio
allocation \(x^t\) belongs to the unit simplex
\[
X
:=
\left\{
z\in\mathbb R^n:
\sum_{i=1}^{n}z_i=1,\;
z_i\geq0,\; i=1,\ldots,n
\right\}.
\]
The set of admissible multi-period allocations is
\[
\mathcal X
:=
X^T
=
\underbrace{X\times\cdots\times X}_{T\text{ times}}.
\]

Let \(w_{\mathrm{prev}}\in X\) denote the portfolio held immediately
before the beginning of the planning horizon. Thus, the first-period
portfolio \(x^1\) is obtained by rebalancing
\(w_{\mathrm{prev}}\).

For each period \(t=1,\ldots,T\), let \(S_t\in\mathbb N\) denote the
number of return scenarios. The random return vector
\(\xi^t\in\mathbb R^n\) is assumed to take the scenario values
\(\xi_j^t\in\mathbb R^n\), \(j=1,\ldots,S_t\), with corresponding
probabilities \(p_j^t\) satisfying
\( 
p_j^t>0,\;
j=1,\ldots,S_t,\;
\sum_{j=1}^{S_t}p_j^t=1.
\) 
For a portfolio \(z\in\mathbb R^n\), the loss under scenario \(j\) at
period \(t\) is
\( 
L_j^t(z)
:=
-(\xi_j^t)^\top z.
\) Thus, the portfolio loss at period \(t\) has the discrete distribution
\( 
\left\{
\bigl(L_j^t(z),p_j^t\bigr)
\right\}_{j=1}^{S_t}.
\)
The expected-return vector is given by
\( 
\mu^t
:=
\mathbb E[\xi^t]
=
\sum_{j=1}^{S_t}p_j^t\xi_j^t.
\)

We consider the following multi-period portfolio problem:
\begin{equation}
\label{P}
\begin{aligned}
\min_{x=(x^1,\ldots,x^T)\in\mathcal X}\;
f(x)
:={}&
-\sum_{t=1}^{T}(\mu^t)^\top x^t
\\
&+
\lambda_1
\left[
\sum_{i=1}^{n}
c_i\left|x_i^1-(w_{\mathrm{prev}})_i\right|
+
\sum_{t=2}^{T}\sum_{i=1}^{n}
c_i\left|x_i^t-x_i^{t-1}\right|
\right]
+
\lambda_2
\sum_{t=1}^{T}\|x^t\|^2
\\[1mm]
\text{\rm s.t.}\;&
\operatorname{VaR}_{\alpha}
\left(-(\xi^t)^\top x^t\right)
\leq\tau,\;
t=1,\ldots,T.
\end{aligned}
\end{equation}
Here, \(\alpha\in(0,1)\) is the VaR confidence level,
\(\lambda_1>0\) is the weight assigned to transaction costs,
\(\lambda_2>0\) is the weight of the quadratic diversification
regularizer, \(c_i\geq0\) is the transaction-cost coefficient of
asset \(i\), and \(\tau\in\mathbb R\) is the maximum permitted VaR
level.

Define the feasible set of Problem~\eqref{P} by
\[
\Omega
:=
\left\{
x=(x^1,\ldots,x^T)\in\mathcal X:
\operatorname{VaR}_{\alpha}
\bigl(-(\xi^t)^\top x^t\bigr)
\leq\tau,\;
t=1,\ldots,T
\right\}.
\]
We also define
\( 
v^*
:=
\inf_{x\in\Omega}f(x)
\) 
and
\( 
S^*
:=
\operatorname*{arg\,min}_{x\in\Omega}f(x).
\)

\begin{theorem}
\label{thm:existence}

Assume that \(\Omega\neq\varnothing\). Then Problem~\eqref{P}
admits a global minimizer. Its optimal value \(v^*\) is finite, and
the global solution set \(S^*\) is nonempty and compact.

\end{theorem}

\begin{proof}

The simplex \(X\) is compact; hence, the Cartesian product
\(\mathcal X=X^T\) is compact.

Applying Lemma~\ref{lem:VaRDC} to each period
\(t=1,\ldots,T\), the mapping
\[
x^t
\longmapsto
\operatorname{VaR}_{\alpha}
\bigl(-(\xi^t)^\top x^t\bigr)
\]
is continuous. Therefore, \(\Omega\) is closed in \(\mathcal X\).
Since \(\mathcal X\) is compact, \(\Omega\) is compact.

The objective function \(f\) is continuous. Hence, by the Weierstrass
theorem, \(f\) attains its minimum on the nonempty compact set
\(\Omega\). Therefore, \(v^*\) is finite and attained, and \(S^*\) is
nonempty.

Moreover,
\[
S^*
=
\left\{
x\in\Omega:
f(x)=v^*
\right\}.
\]
The continuity of \(f\) implies that \(S^*\) is closed in \(\Omega\).
Consequently, \(S^*\) is compact.

\end{proof}
\subsection{Penalized DC Reformulation}
\label{subsec:penalized_dc}

We now construct a penalized DC formulation associated with
Problem~\eqref{P}. For \(x\in\mathcal X\), define the aggregate
VaR-constraint violation by
\begin{equation}
\label{eq:R_def}
R(x)
:=
\sum_{t=1}^{T}
\left(
\operatorname{VaR}_{\alpha}
\bigl(-(\xi^t)^\top x^t\bigr)-\tau
\right)_+.
\end{equation}
Each term in the definition of \(R(x)\) is nonnegative and measures
the violation of the corresponding VaR constraint. Therefore,
\( 
R(x)\geq0,\;
\forall x\in\mathcal X.
\) 
Moreover,
\(
R(x)=0
\)
if and only if
\(
x\in\Omega.
\)
For a penalty parameter \(\rho>0\), define
\begin{equation}
\label{eq:Phi_def}
\Phi_\rho(x)
:=
f(x)+\rho R(x),
\; x\in\mathcal X,
\end{equation}
and consider the penalized problem
\begin{equation}
\label{eq:penalty_problem}
\min_{x\in\mathcal X}\;
\Phi_\rho(x).
\end{equation}

To derive a DC decomposition of \(\Phi_\rho\), for
\(\beta\in(0,1)\) and \(t=1,\ldots,T\), define
\[
\psi_{\beta,t}(z)
:=
\operatorname{CVaR}_{\beta}
\bigl(-(\xi^t)^\top z\bigr).
\]
For the finite-scenario distribution introduced above, its
Rockafellar--Uryasev representation is
\begin{equation}
\label{eq:cvar_finite}
\psi_{\beta,t}(z)
=
\min_{u\in\mathbb R}
\left\{
u+
\frac{1}{1-\beta}
\sum_{j=1}^{S_t}p_j^t
\bigl(L_j^t(z)-u\bigr)_+
\right\};
\end{equation}
see
\cite{RockafellarUryasev2000,RockafellarUryasev2002}.
Consequently, \(\psi_{\beta,t}\) is a finite-valued convex
piecewise-affine function of \(z\).

For each period \(t\), define
\[
\bar\varepsilon_t
:=
\alpha-
\max_{\substack{
J\subseteq\{1,\ldots,S_t\}\\
\sum_{j\in J}p_j^t<\alpha
}}
\sum_{j\in J}p_j^t.
\]
By Lemma~\ref{lem:VaRDC},
\(
0<\bar\varepsilon_t\leq\alpha
\)
for every \(t=1,\ldots,T\). Since \(T\) is finite,
\(
\min_{1\leq t\leq T}\bar\varepsilon_t>0.
\)
Choose
\begin{equation}
\label{eq:gamma_choice}
0<\gamma
<
\min_{1\leq t\leq T}\bar\varepsilon_t
\end{equation}
and set
\( 
A_\gamma
:=
\frac{1-\alpha+\gamma}{\gamma},\;
B_\gamma
:=
\frac{1-\alpha}{\gamma}.
\) 
Then \(0<\gamma<\alpha\), and hence
\(\alpha-\gamma\in(0,1)\). Applying
Lemma~\ref{lem:VaRDC} to each planning period gives
\begin{equation}
\label{eq:loss_var_dc}
\operatorname{VaR}_{\alpha}
\bigl(-(\xi^t)^\top z\bigr)
=
A_\gamma\psi_{\alpha-\gamma,t}(z)
-
B_\gamma\psi_{\alpha,t}(z),
\;
t=1,\ldots,T,\; z\in\mathbb R^n.
\end{equation}

Substituting~\eqref{eq:loss_var_dc} into~\eqref{eq:R_def} gives
\[
R(x)
=
\sum_{t=1}^{T}
\left(
A_\gamma\psi_{\alpha-\gamma,t}(x^t)
-
B_\gamma\psi_{\alpha,t}(x^t)
-\tau
\right)_+.
\]
Using the identity
\( 
(a-b)_+
=
\max\{a,b\}-b,
\)
the penalized objective admits the DC decomposition
\[
\Phi_\rho(x)
=
G_\rho(x)-H_\rho(x),
\]
where
\begin{align}
G_\rho(x)
:={}&
f(x)
+
\rho\sum_{t=1}^{T}
\max\left\{
A_\gamma\psi_{\alpha-\gamma,t}(x^t)-\tau,\,
B_\gamma\psi_{\alpha,t}(x^t)
\right\},
\label{eq:G_def}
\\
H_\rho(x)
:={}&
\rho B_\gamma
\sum_{t=1}^{T}\psi_{\alpha,t}(x^t).
\label{eq:H_def}
\end{align}
Therefore, the penalized problem associated with
Problem~\eqref{P} has the DC formulation
\begin{equation}
\label{DCP}
\min_{x\in\mathcal X}\;
G_\rho(x)-H_\rho(x).
\end{equation}

Both \(G_\rho\) and \(H_\rho\) are finite-valued convex functions.
Moreover, \(G_\rho\) is strongly convex with modulus at least
\(2\lambda_2\), since \(f\) contains the term
\( 
\lambda_2\sum_{t=1}^{T}\|x^t\|^2
=
\lambda_2\|x\|^2.
\)

We next give a sufficient condition under which the penalized
Problem~\eqref{eq:penalty_problem} is globally equivalent to the
original constrained Problem~\eqref{P}. For a nonempty set
\(C\subseteq\mathbb R^{nT}\), define
\[
\operatorname{dist}(x,C)
:=
\inf_{y\in C}\|x-y\|.
\]

\begin{lemma}
\label{lem:global_error_bound}
Assume that \(\Omega\neq\varnothing\). Then there exists a constant
\(\kappa>0\) such that
\[
\operatorname{dist}(x,\Omega)
\leq
\kappa R(x),
\;
\forall x\in\mathcal X.
\]
\end{lemma}

\begin{proof}
The function \(R\) is a nonnegative continuous piecewise-affine
function on the compact polyhedron \(\mathcal X\). Hence, there
exists a finite collection of nonempty compact polyhedra
\(\{P_\ell\}_{\ell=1}^{N}\) covering \(\mathcal X\) such that, for
each \(\ell\),
\[
R(x)=a_\ell^\top x+b_\ell,
\; x\in P_\ell,
\]
for some \(a_\ell\in\mathbb R^{nT}\) and \(b_\ell\in\mathbb R\).

Fix \(\ell\). Suppose first that
\(P_\ell\cap\Omega\neq\varnothing\). Since
\(R(x)=0\) if and only if \(x\in\Omega\),
\[
P_\ell\cap\Omega
=
\left\{
x\in P_\ell:
a_\ell^\top x+b_\ell=0
\right\}.
\]
By the Hoffman-type global error bound for polyhedral systems
\cite[Example~9.47]{RockafellarWets1998}, there exists
\(\kappa_\ell>0\) such that
\[
\operatorname{dist}
\bigl(x,P_\ell\cap\Omega\bigr)
\leq
\kappa_\ell
\left|a_\ell^\top x+b_\ell\right|,
\;
\forall x\in P_\ell.
\]
Since \(R\geq0\) and
\(R(x)=a_\ell^\top x+b_\ell\) on \(P_\ell\),
\(
\left|a_\ell^\top x+b_\ell\right|=R(x).
\)
Moreover, \(P_\ell\cap\Omega\subseteq\Omega\). Therefore,
\begin{equation}
\label{v_5aug_4}
\operatorname{dist}(x,\Omega)
\leq
\operatorname{dist}
\bigl(x,P_\ell\cap\Omega\bigr)
\leq
\kappa_\ell R(x),
\;
\forall x\in P_\ell.
\end{equation}

Suppose now that \(P_\ell\cap\Omega=\varnothing\). Then \(R(x)>0\)
for every \(x\in P_\ell\). By compactness and continuity,
\[
m_\ell
:=
\min_{x\in P_\ell}R(x)
>0,\;
D_\ell
:=
\max_{x\in P_\ell}
\operatorname{dist}(x,\Omega)
<+\infty.
\]

Since \(R(x)\geq m_\ell\) and
\(\operatorname{dist}(x,\Omega)\leq D_\ell\) on \(P_\ell\), we have
\begin{equation}
\label{v_5aug_5}
\operatorname{dist}(x,\Omega)
\leq
D_\ell
\leq
\frac{D_\ell}{m_\ell}R(x),
\;
\forall x\in P_\ell.
\end{equation}

For each \(\ell\), define
\[
c_\ell
:=
\begin{cases}
\kappa_\ell,
& P_\ell\cap\Omega\neq\varnothing,\\[1mm]
D_\ell/m_\ell,
& P_\ell\cap\Omega=\varnothing.
\end{cases}
\]
By~\eqref{v_5aug_4} and~\eqref{v_5aug_5},
\(
\operatorname{dist}(x,\Omega)
\leq
c_\ell R(x),
\;
\forall x\in P_\ell.
\)
Since the covering is finite, setting
\(
\kappa
:=
\max_{1\leq\ell\leq N}c_\ell
\)
gives
\(
\operatorname{dist}(x,\Omega)
\leq
\kappa R(x),
\;
\forall x\in\mathcal X.
\)
\end{proof}

\begin{theorem}
\label{thm:exact_penalty}
Assume that \(\Omega\neq\varnothing\). Let \(L_f>0\) be a Lipschitz
constant of \(f\) on \(\mathcal X\), and let \(\kappa>0\) be given by
Lemma~\ref{lem:global_error_bound}. If
\( 
\rho>L_f\kappa,
\) 
then Problem~\eqref{P} and the penalized
Problem~\eqref{eq:penalty_problem} have the same optimal value and
the same set of global minimizers.
\end{theorem}

\begin{proof}
Let \(x^*\in S^*\), which exists by
Theorem~\ref{thm:existence}. Fix \(x\in\mathcal X\), and choose
\(y\in\Omega\) such that
\( 
\|x-y\|
=
\operatorname{dist}(x,\Omega),
\) 
which exists because \(\Omega\) is nonempty and compact. By the
Lipschitz continuity of \(f\) and
Lemma~\ref{lem:global_error_bound},
\[
f(x)
\geq
f(y)-L_f\|x-y\|
\geq
f(y)-L_f\kappa R(x).
\]
Since \(y\in\Omega\), we have \(f(y)\geq v^*\). Therefore,
\begin{equation}
\label{eq:penalty_lower_bound}
\Phi_\rho(x)
=
f(x)+\rho R(x)
\geq
v^*+(\rho-L_f\kappa)R(x),
\;
\forall x\in\mathcal X.
\end{equation}

We first show that \(x^*\) is a global minimizer of the penalized
problem. Since \(x^*\in S^*\subseteq\Omega\),
\[
R(x^*)=0
\;\text{and}\;
\Phi_\rho(x^*)
=
f(x^*)
=
v^*.
\]
Moreover, since \(\rho>L_f\kappa\) and \(R(x)\geq0\),
\eqref{eq:penalty_lower_bound} gives
\(
\Phi_\rho(x)\geq v^*,\;
\forall x\in\mathcal X.
\)
Hence, \(x^*\) is a global minimizer of the penalized problem.

Conversely, let \(\bar x\) be a global minimizer of the penalized
problem. Since \(x^*\in\mathcal X\), the optimality of \(\bar x\)
gives
\( 
\Phi_\rho(\bar x)
\leq
\Phi_\rho(x^*)
=
v^*.
\) 
Suppose that \(\bar x\notin\Omega\). Then \(R(\bar x)>0\), and
\eqref{eq:penalty_lower_bound}, together with
\(\rho>L_f\kappa\), yields
\[
\Phi_\rho(\bar x)
\geq
v^*+(\rho-L_f\kappa)R(\bar x)
>
v^*,
\]
which contradicts \(\Phi_\rho(\bar x)\leq v^*\). Therefore,
\(\bar x\in\Omega\), and hence \(R(\bar x)=0\). It follows that
\(
f(\bar x)
=
\Phi_\rho(\bar x)
\leq
v^*.
\)
Since \(\bar x\in\Omega\) and \(v^*\) is the minimum of \(f\) over
\(\Omega\), we also have \(f(\bar x)\geq v^*\). Thus,
\( 
f(\bar x)=v^*,
\) 
and consequently \(\bar x\in S^*\). 

Therefore, the original and
penalized problems have the same set of global minimizers. Since the
penalty vanishes at every point of this set, both problems also have
the same optimal value \(v^*\).
\end{proof}

\section{Projected Inertial BDCA and Convergence Analysis}
\label{sec4}

In this section, we develop a projected inertial boosted
difference-of-convex algorithm for solving the penalized DC
Problem~\eqref{DCP}. The proposed method builds on the BDCA framework
for single-period VaR-constrained portfolio optimization developed by
Thormann et al.~\cite{ThormannVuongZemkoho2026}. We first construct the
strongly convex DC decomposition and the CVaR subgradient required to
implement the method. We then analyze its well-definedness, descent
properties, and convergence to critical points. Finally, under an
additional local no-ties condition, we establish whole-sequence and
\(R\)-linear convergence.
\subsection{Projected Inertial BDCA}
\label{subsec:proposed_algorithm}

We first introduce an equivalent DC decomposition with strongly convex
components. Fix \(\nu>0\) and define
\begin{equation}
\label{eq:regularized_dc}
\widehat G_\rho(x)
:=
G_\rho(x)+\frac{\nu}{2}\|x\|^2,
\;
\widehat H_\rho(x)
:=
H_\rho(x)+\frac{\nu}{2}\|x\|^2.
\end{equation}
Since the same quadratic term is added to both components,
\begin{equation}
\label{eq:equivalent_dc}
\Phi_\rho
=
G_\rho-H_\rho
=
\widehat G_\rho-\widehat H_\rho.
\end{equation}
Moreover, \(\widehat G_\rho\) and \(\widehat H_\rho\) are strongly
convex with modulus \(2\lambda_2+\nu\) and \(\nu\), respectively.
Since \(\mathcal X\) is nonempty, compact, and convex, for every
\(u\in\mathbb R^{nT}\), the convex problem
\[
\min_{x\in\mathcal X}
\left\{
\widehat G_\rho(x)-\langle u,x\rangle
\right\}
\]
has a unique solution. Furthermore,
\[
\partial\widehat G_\rho(x)
=
\partial G_\rho(x)+\nu x,
\;
\partial\widehat H_\rho(x)
=
\partial H_\rho(x)+\nu x.
\]
Therefore,
\begin{equation}
\label{eq:criticality_equivalence}
\begin{aligned}
&\partial\widehat H_\rho(x)
\cap
\bigl(
\partial\widehat G_\rho(x)+N_{\mathcal X}(x)
\bigr)
\neq\varnothing
\;\Longleftrightarrow\;
\partial H_\rho(x)
\cap
\bigl(
\partial G_\rho(x)+N_{\mathcal X}(x)
\bigr)
\neq\varnothing.
\end{aligned}
\end{equation}
Hence, the two decompositions have the same critical points.

We next derive an explicit scenario-based subgradient formula for
\(H_\rho\). 
\begin{lemma}[Finite-scenario CVaR subgradient]
\label{lem:subCVaR}

Fix \(t\in\{1,\ldots,T\}\), \(\beta\in(0,1)\), and
\(z\in\mathbb R^n\). Let \(u^*\) be any solution of the scalar
problem in~\eqref{eq:cvar_finite}. Then there exist
\(\vartheta_j\), \(j=1,\ldots,S_t\), satisfying
\begin{equation}
\label{eq:theta_conditions}
\vartheta_j\in
\begin{cases}
\{1\}, & L_j^t(z)>u^*,\\
[0,1], & L_j^t(z)=u^*,\\
\{0\}, & L_j^t(z)<u^*,
\end{cases}
\;
\sum_{j=1}^{S_t}p_j^t\vartheta_j=1-\beta.
\end{equation}
For any such choice,
\begin{equation}
\label{eq:cvar_subgradient}
g_{\beta,t}(z)
:=
-\frac{1}{1-\beta}
\sum_{j=1}^{S_t}
p_j^t\vartheta_j\xi_j^t
\in
\partial\psi_{\beta,t}(z).
\end{equation}
Moreover,
\( 
\partial H_\rho(x)
=
\rho B_\gamma
\left(
\partial\psi_{\alpha,1}(x^1)
\times\cdots\times
\partial\psi_{\alpha,T}(x^T)
\right).
\)

\end{lemma}

\begin{proof}

Define
\[
F_{\beta,t}(z,u)
:=
u+
\frac{1}{1-\beta}
\sum_{j=1}^{S_t}
p_j^t\bigl(L_j^t(z)-u\bigr)_+.
\]
Then
\( 
\psi_{\beta,t}(z)
=
\min_{u\in\mathbb R}F_{\beta,t}(z,u).
\) Recall that
\[
\partial[a]_+
=
\begin{cases}
\{1\}, & a>0,\\
[0,1], & a=0,\\
\{0\}, & a<0.
\end{cases}
\]
For each \(j\), set
\( 
\Theta_j
:=
\partial
\bigl[L_j^t(z)-u^*\bigr]_+.
\)
Thus,
\[
\Theta_j
=
\begin{cases}
\{1\}, & L_j^t(z)>u^*,\\
[0,1], & L_j^t(z)=u^*,\\
\{0\}, & L_j^t(z)<u^*.
\end{cases}
\]

Since \(u^*\) minimizes \(F_{\beta,t}(z,\cdot)\), its optimality
condition is
\[
0
\in
\partial_u F_{\beta,t}(z,u^*)
=
1-
\frac{1}{1-\beta}
\sum_{j=1}^{S_t}p_j^t\Theta_j.
\]
Therefore, there exist
\(\vartheta_j\in\Theta_j\) such that
\( 
1-
\frac{1}{1-\beta}
\sum_{j=1}^{S_t}p_j^t\vartheta_j
=
0,
\) 
or equivalently,
\( 
\sum_{j=1}^{S_t}p_j^t\vartheta_j
=
1-\beta.
\) 
This proves the existence of the coefficients satisfying
\eqref{eq:theta_conditions}.

Now fix any such choice of \(\vartheta_j\). Since
\(
L_j^t(z)-u
=
-(\xi_j^t)^\top z-u,
\)
the subgradient of this affine expression with respect to the pair
\((z,u)\) is
\(
\bigl(-\xi_j^t,-1\bigr).
\) 
Hence,
\[
\vartheta_j
\bigl(-\xi_j^t,-1\bigr)
\in
\partial_{(z,u)}
\bigl(L_j^t(z)-u\bigr)_+
\big|_{u=u^*}.
\]
Using the convex subdifferential sum rule, we obtain
\[
\left(
-\frac{1}{1-\beta}
\sum_{j=1}^{S_t}p_j^t\vartheta_j\xi_j^t,\,
1-\frac{1}{1-\beta}
\sum_{j=1}^{S_t}p_j^t\vartheta_j
\right)
\in
\partial F_{\beta,t}(z,u^*).
\]
By~\eqref{eq:cvar_subgradient}, the first component is
\(g_{\beta,t}(z)\), while the mass condition in
\eqref{eq:theta_conditions} makes the second component equal to zero.
Therefore,
\(
\bigl(g_{\beta,t}(z),0\bigr)
\in
\partial F_{\beta,t}(z,u^*).
\) By the definition of the convex subdifferential, for every
\(y\in\mathbb R^n\) and \(v\in\mathbb R\),
\[
F_{\beta,t}(y,v)
\geq
F_{\beta,t}(z,u^*)
+
\left\langle
\bigl(g_{\beta,t}(z),0\bigr),
(y-z,v-u^*)
\right\rangle.
\]
Thus,
\( 
F_{\beta,t}(y,v)
\geq
F_{\beta,t}(z,u^*)
+
\langle g_{\beta,t}(z),y-z\rangle.
\) 
Taking the infimum over \(v\in\mathbb R\) yields
\[
\psi_{\beta,t}(y)
\geq
\psi_{\beta,t}(z)
+
\langle g_{\beta,t}(z),y-z\rangle.
\]
Hence,
\(
g_{\beta,t}(z)
\in
\partial\psi_{\beta,t}(z).
\)  Finally, the formula for \(\partial H_\rho\) follows directly from
\eqref{eq:H_def} and the subdifferential rule for separable convex
functions.

\end{proof}
We now present the proposed projected inertial BDCA, where
\(\Pi_{\mathcal X}\) denotes the Euclidean projection onto
\(\mathcal X\).

\begin{algorithm}[H]
\caption{Projected Inertial Boosted Difference-of-Convex Algorithm}
\label{MainAlgo}

\begin{algorithmic}[1]

\State \textbf{Input:}
\(\rho>0\),
\(x^{-1}=x^0\in\mathcal X\),
\(\theta\in[0,1)\),
\(\nu>0\),
\(\sigma\in(0,\lambda_2+\nu)\),
\(\eta\in(0,1)\),
\(\bar\lambda>1\), and
\(\varepsilon_{\mathrm{tol}}\geq0\).

\For{\(k=0,1,2,\ldots\)}

\State Set
\[
w^k
=
\Pi_{\mathcal X}
\bigl(
x^k+\theta(x^k-x^{k-1})
\bigr).
\]

\If{\(\Phi_\rho(w^k)>\Phi_\rho(x^k)\)}
\State Set \(w^k=x^k\).
\EndIf

\State For each \(t=1,\ldots,T\), compute
\[
g_{\alpha,t}(w^{k,t})
\in
\partial\psi_{\alpha,t}(w^{k,t})
\]
using~\eqref{eq:theta_conditions} and
\eqref{eq:cvar_subgradient}.\State Set
\[
s^k
=
\rho B_\gamma
\bigl(
g_{\alpha,1}(w^{k,1}),\ldots,
g_{\alpha,T}(w^{k,T})
\bigr)
\in
\partial H_\rho(w^k)
\]
and
\[
u^k
=
s^k+\nu w^k
\in
\partial\widehat H_\rho(w^k).
\]

\State Compute the unique solution
\[
y^k
=
\operatorname*{arg\,min}_{x\in\mathcal X}
\left\{
\widehat G_\rho(x)-\langle u^k,x\rangle
\right\}.
\]

\State Set \(d^k=y^k-w^k\).

\If{\(\|d^k\|=0\)}
\State \textbf{return} \(w^k\).
\EndIf

\If{\(\varepsilon_{\mathrm{tol}}>0\) and
\(\|d^k\|
\leq
\varepsilon_{\mathrm{tol}}
\max\{1,\|w^k\|\}\)}
\State \textbf{return} \(w^k\).
\EndIf

\State Compute
\[
\lambda_k^{\max}
=
\max
\left\{
\lambda\geq1:
w^k+\lambda d^k\in\mathcal X
\right\}.
\]

\State Set
\[
\lambda_k
=
\min\{\bar\lambda,\lambda_k^{\max}\}.
\]

\While{\(\lambda_k>1\) and
\(\Phi_\rho(w^k+\lambda_kd^k)>
\min\{
\Phi_\rho(y^k),
\Phi_\rho(w^k)
-\sigma\lambda_k^2\|d^k\|^2
\}\)}

\State Set
\[
\lambda_k
=
\max\{1,\eta\lambda_k\}.
\]

\EndWhile

\State Set
\[
x^{k+1}
=
w^k+\lambda_kd^k.
\]

\EndFor

\end{algorithmic}
\end{algorithm}

\subsection{Convergence Analysis}
\label{subsec:convergence_analysis}

We now analyze the convergence properties of
Algorithm~\ref{MainAlgo}. We first prove that the method is well
defined, that the generated objective values decrease monotonically,
and that every accumulation point is critical. We then establish
whole-sequence and \(R\)-linear convergence under the local no-ties
condition. 

Define the indicator function of \(\mathcal X\) by
\[
\delta_{\mathcal X}(x)
:=
\begin{cases}
0,       & x\in\mathcal X,\\
+\infty, & x\notin\mathcal X.
\end{cases}
\]
For brevity, write
\[
\widehat Q_\rho
:=
\widehat G_\rho+\delta_{\mathcal X}.
\]
Since \(\widehat G_\rho\) is finite and continuous on
\(\mathbb R^{nT}\), the convex subdifferential sum rule gives
\[
\partial\widehat Q_\rho(x)
=
\partial\widehat G_\rho(x)+N_{\mathcal X}(x),
\; x\in\mathcal X.
\]

\begin{lemma}\label{lem:welldefined}
At every iteration before termination, Algorithm~\ref{MainAlgo} is
well defined. More precisely:
\begin{enumerate}
\item[\rm(i)]
the convex subproblem defining \(y^k\) has a unique solution;
\item[\rm(ii)]
\[
\Phi_\rho(y^k)
\leq
\Phi_\rho(w^k)-(\lambda_2+\nu)\|d^k\|^2;
\]
\item[\rm(iii)]
if the algorithm reaches the line-search step, then
\(\lambda_k^{\max}\) exists, satisfies
\(\lambda_k^{\max}\geq1\), and the Armijo backtracking terminates after
finitely many reductions with
\[
1\leq\lambda_k
\leq\min\{\bar\lambda,\lambda_k^{\max}\}.
\]
Consequently, \(x^{k+1}\in\mathcal X\).\end{enumerate}
\end{lemma}
\begin{proof}

The function
\( 
x\longmapsto
\widehat G_\rho(x)-\langle u^k,x\rangle
\) 
is continuous and \((2\lambda_2+\nu)\)-strongly convex. Since
\(\mathcal X\) is nonempty and compact, the subproblem defining
\(y^k\) has a minimizer. Since \(\mathcal X\) is also convex and the
objective is strongly convex, this minimizer is unique. This proves
\({\rm(i)}\).

\medskip
The optimality condition for the convex subproblem is
\[
0
\in
\partial\widehat G_\rho(y^k)
-u^k
+N_{\mathcal X}(y^k).
\]
Hence, there exists \(v^k\in N_{\mathcal X}(y^k)\) such that
\( 
u^k-v^k
\in
\partial\widehat G_\rho(y^k).
\) 
Since \(\widehat G_\rho\) is
\((2\lambda_2+\nu)\)-strongly convex,
\begin{equation} \label{v_3aug_1}
\begin{aligned}
\widehat G_\rho(w^k)-\widehat G_\rho(y^k)
\geq{}&
\langle u^k-v^k,w^k-y^k\rangle
+ 
\left(\lambda_2+\frac{\nu}{2}\right)
\|w^k-y^k\|^2.
\end{aligned}
\end{equation}
On the other hand,
\(u^k\in\partial\widehat H_\rho(w^k)\), and
\(\widehat H_\rho\) is \(\nu\)-strongly convex. Therefore,
\begin{equation}\label{v_3aug_2}
\widehat H_\rho(y^k)-\widehat H_\rho(w^k)
\geq
\langle u^k,y^k-w^k\rangle
+
\frac{\nu}{2}\|y^k-w^k\|^2.
\end{equation}
Adding the inequalities~\eqref{v_3aug_1}--\eqref{v_3aug_2}, and using
\(
\Phi_\rho
=
\widehat G_\rho-\widehat H_\rho
\)
gives
\[
\begin{aligned}
\Phi_\rho(w^k)-\Phi_\rho(y^k)
\geq{}&
-\langle v^k,w^k-y^k\rangle
+
(\lambda_2+\nu)\|w^k-y^k\|^2.
\end{aligned}
\]
Moreover, \(v^k\in N_{\mathcal X}(y^k)\) and
\(w^k\in\mathcal X\), so
\(
-\langle v^k,w^k-y^k\rangle\geq0.
\)
Since \(d^k=y^k-w^k\), we obtain
\[
\Phi_\rho(y^k)
\leq
\Phi_\rho(w^k)
-
(\lambda_2+\nu)\|d^k\|^2.
\]
This proves \({\rm(ii)}\).

We finally prove \({\rm(iii)}\). Suppose that \(d^k\neq0\) and that
the algorithm proceeds to the line-search step. Recall that
\(\mathcal X=X^T\). Hence,
\[
w^k+\lambda d^k\in\mathcal X
\;\Longleftrightarrow\;
w^{k,t}+\lambda d^{k,t}\in X,
\;
t=1,\ldots,T.
\]
Therefore, to determine \(\lambda_k^{\max}\), it suffices to verify,
for every \(t=1,\ldots,T\), that
\[
\sum_{i=1}^{n}
\bigl(w_i^{k,t}+\lambda d_i^{k,t}\bigr)=1
\quad\text{and}\quad
w_i^{k,t}+\lambda d_i^{k,t}\geq0,
\;
i=1,\ldots,n.
\]

Since \(w^{k,t},y^{k,t}\in X\) and
\(d^{k,t}=y^{k,t}-w^{k,t}\), for every \(t=1,\ldots,T\),
\[
\sum_{i=1}^{n}d_i^{k,t}
=
\sum_{i=1}^{n}y_i^{k,t}
-
\sum_{i=1}^{n}w_i^{k,t}
=
0.
\]
Hence, for every \(\lambda\in\mathbb R\),
\(
\sum_{i=1}^{n}
\bigl(w_i^{k,t}+\lambda d_i^{k,t}\bigr)
=
1.
\)
Moreover, if \(d_i^{k,t}\geq0\), then
\(
w_i^{k,t}+\lambda d_i^{k,t}\geq0,
\;
\lambda\geq0.
\)
If \(d_i^{k,t}<0\), then
\[
w_i^{k,t}+\lambda d_i^{k,t}\geq0
\quad\Longleftrightarrow\quad
\lambda
\leq
-\frac{w_i^{k,t}}{d_i^{k,t}}.
\]

Since \(d^k\neq0\) and each block of \(d^k\) has zero sum, \(d^k\)
has at least one negative component. Therefore,
\[
\lambda_k^{\max}
=
\min_{\substack{
t=1,\ldots,T,\;i=1,\ldots,n\\
d_i^{k,t}<0
}}
\left\{
-\frac{w_i^{k,t}}{d_i^{k,t}}
\right\}
\]
is well defined and finite. Since \(w^k+d^k=y^k\in\mathcal X\), the value \(\lambda=1\) is
feasible, and hence \(\lambda_k^{\max}\geq1\). Because
\(\bar\lambda>1\), the initial trial step size
\[
\lambda_k=\min\{\bar\lambda,\lambda_k^{\max}\}
\]
belongs to \([1,\lambda_k^{\max}]\). Each backtracking update
\(\lambda_k\leftarrow\max\{1,\eta\lambda_k\}\) preserves this
inclusion. Hence, the corresponding trial point
\(w^k+\lambda_kd^k\) belongs to \(\mathcal X\).

At \(\lambda_k=1\), we have \(w^k+\lambda_kd^k=y^k\). Moreover, by
\({\rm(ii)}\) and \(\sigma<\lambda_2+\nu\),
\[
\Phi_\rho(y^k)
\leq
\Phi_\rho(w^k)
-
(\lambda_2+\nu)\|d^k\|^2
\leq
\Phi_\rho(w^k)-\sigma\|d^k\|^2.
\]
Therefore,
\[
\Phi_\rho(w^k+d^k)
=
\min\left\{
\Phi_\rho(y^k),
\Phi_\rho(w^k)-\sigma\|d^k\|^2
\right\},
\]
so \(\lambda_k=1\) satisfies the line-search condition. Since
\(0<\eta<1\), the update reaches \(\lambda_k=1\) after finitely many
reductions whenever no larger step is accepted. Consequently,
\[
1\leq\lambda_k
\leq
\min\{\bar\lambda,\lambda_k^{\max}\},
\;
x^{k+1}=w^k+\lambda_kd^k\in\mathcal X.
\]
This proves \({\rm(iii)}\) and completes the proof.\end{proof}
\begin{proposition}\label{prop:exact_stop}
If \(d^k=0\) at some iteration, then \(w^k\) is a critical point of
Problem~\eqref{DCP}.
\end{proposition}
\begin{proof}
If \(d^k=0\), then \(y^k=w^k\). The optimality condition of the
convex DCA subproblem gives
\[
u^k
\in
\partial\widehat G_\rho(w^k)
+
N_{\mathcal X}(w^k).
\]
By the definition of \(u^k\) in Algorithm~\ref{MainAlgo},
\(u^k\in\partial\widehat H_\rho(w^k)\). Therefore, \(w^k\) is critical for the regularized DC decomposition.
By~\eqref{eq:criticality_equivalence}, it is also critical for the
problem in~\eqref{DCP}.
\end{proof}

\begin{proposition}\label{prop:descent}
Let \(\{x^k\}\) be an infinite sequence generated by
Algorithm~\ref{MainAlgo}. Then
\begin{equation}\label{eq:descent}
\Phi_\rho(x^{k+1})
\leq
\Phi_\rho(w^k)
-\sigma\lambda_k^2\|d^k\|^2
\leq
\Phi_\rho(x^k)
-\sigma\lambda_k^2\|d^k\|^2.
\end{equation}
In addition,
\begin{equation}\label{eq:bracketing}
\Phi_\rho(w^{k+1})
\leq
\Phi_\rho(x^{k+1})
\leq
\Phi_\rho(y^k)
\leq
\Phi_\rho(w^k),
\end{equation}
and
\begin{equation}\label{eq:w_descent}
\Phi_\rho(w^{k+1})
\leq
\Phi_\rho(w^k)
-\sigma\lambda_k^2\|d^k\|^2.
\end{equation}
Consequently, both \(\{\Phi_\rho(x^k)\}\) and
\(\{\Phi_\rho(w^k)\}\) are nonincreasing, and
\[
\sum_{k=0}^{\infty}
\lambda_k^2\|d^k\|^2<+\infty,
\;
\sum_{k=0}^{\infty}\|d^k\|^2<+\infty.
\]
\end{proposition}

\begin{proof}

By the Armijo acceptance condition,
\(x^{k+1}-w^k=\lambda_kd^k\), and the inertial safeguard,
\[
\Phi_\rho(x^{k+1})
\leq
\Phi_\rho(w^k)-\sigma\lambda_k^2\|d^k\|^2
\leq
\Phi_\rho(x^k)-\sigma\lambda_k^2\|d^k\|^2.
\]
This proves~\eqref{eq:descent} and shows that
\(\{\Phi_\rho(x^k)\}\) is nonincreasing.
\medskip

By the line-search condition and
Lemma~\ref{lem:welldefined}\({\rm(ii)}\),
\[
\Phi_\rho(x^{k+1})
\leq
\Phi_\rho(y^k)
\leq
\Phi_\rho(w^k).
\]
At iteration \(k+1\), the inertial safeguard gives
\( 
\Phi_\rho(w^{k+1})
\leq
\Phi_\rho(x^{k+1}).
\) 
Combining these inequalities yields
\[
\Phi_\rho(w^{k+1})
\leq
\Phi_\rho(x^{k+1})
\leq
\Phi_\rho(y^k)
\leq
\Phi_\rho(w^k),
\]
which proves~\eqref{eq:bracketing}.

Furthermore, combining
\( 
\Phi_\rho(w^{k+1})
\leq
\Phi_\rho(x^{k+1})
\) 
with the first inequality in~\eqref{eq:descent} gives
\[
\Phi_\rho(w^{k+1})
\leq
\Phi_\rho(w^k)
-
\sigma\lambda_k^2\|d^k\|^2.
\]
This proves~\eqref{eq:w_descent}. In particular,
\(\{\Phi_\rho(w^k)\}\) is nonincreasing.

\medskip

\noindent 
Since \(\Phi_\rho\) is continuous and \(\mathcal X\) is compact,
\(\Phi_\rho\) is bounded below on \(\mathcal X\). Summing~\eqref{eq:descent} from \(k=0\) to \(N\) gives
\[
\sigma
\sum_{k=0}^{N}
\lambda_k^2\|d^k\|^2
\leq
\Phi_\rho(x^0)-\Phi_\rho(x^{N+1})
\leq
\Phi_\rho(x^0)
-
\inf_{x\in\mathcal X}\Phi_\rho(x).
\]
Letting \(N\to\infty\) and using \(\lambda_k\geq1\), we obtain
\[
\sum_{k=0}^{\infty}\|d^k\|^2
\leq
\sum_{k=0}^{\infty}\lambda_k^2\|d^k\|^2
<+\infty.
\]
This completes the proof.
\end{proof}
\begin{proposition}
\label{prop:diff}

Suppose that Algorithm~\ref{MainAlgo} generates an infinite sequence
\(\{x^k\}\), with associated sequences \(\{w^k\}\) and
\(\{y^k\}\). Then
\[
d^k\to0,
\;
\|x^{k+1}-x^k\|\to0,
\;
\|w^k-x^k\|\to0.
\]
Consequently, \(\|y^k-x^k\|\to0\).

\end{proposition}

\begin{proof}

By Proposition~\ref{prop:descent},
\( 
\sum_{k=0}^{\infty}\|d^k\|^2<+\infty.
\)
Therefore, \(d^k=y^k-w^k\to0\). Since
\(\lambda_k\leq\bar\lambda\), it also follows that
\(\lambda_k\|d^k\|\to0\). When the inertial safeguard is inactive,
\[
w^k
=
\Pi_{\mathcal X}
\bigl(
x^k+\theta(x^k-x^{k-1})
\bigr).
\]
Since \(x^k\in\mathcal X\), we have
\(\Pi_{\mathcal X}(x^k)=x^k\). By the nonexpansiveness of the
Euclidean projection,
\[
\begin{aligned}
\|w^k-x^k\|
&=
\left\|
\Pi_{\mathcal X}
\bigl(
x^k+\theta(x^k-x^{k-1})
\bigr)
-
\Pi_{\mathcal X}(x^k)
\right\|
\leq
\theta\|x^k-x^{k-1}\|.
\end{aligned}
\]
When the safeguard is active, \(w^k=x^k\), so the same inequality
holds. Hence, in both cases,
\begin{equation}
\label{eq:inertial_difference}
\|w^k-x^k\|
\leq
\theta\|x^k-x^{k-1}\|.
\end{equation}

Using \(x^{k+1}=w^k+\lambda_kd^k\), we obtain
\[
\|x^{k+1}-x^k\|
\leq
\theta\|x^k-x^{k-1}\|
+
\lambda_k\|d^k\|.
\]
Applying Lemma~\ref{lem:sequence} with
\(a_k:=\|x^k-x^{k-1}\|\) and
\(b_k:=\lambda_k\|d^k\|\), where \(b_k\to0\), gives
\(\|x^{k+1}-x^k\|\to0\). It then follows from
\eqref{eq:inertial_difference} that
\(\|w^k-x^k\|\to0\).

Finally, since \(y^k-w^k=d^k\),
\[
\|y^k-x^k\|
\leq
\|d^k\|+\|w^k-x^k\|
\to0.
\]

\end{proof}
We next establish the boundedness of the subgradients used by the
algorithm. We first recall that the linear-programming dual
of~\eqref{eq:cvar_finite} is
\begin{equation}
\label{eq:cvar_dual}
\psi_{\beta,t}(z)
=
\max_{q\in Q_{\beta,t}}
\left\{
-\sum_{j=1}^{S_t}q_j(\xi_j^t)^\top z
\right\},
\end{equation}
where
\[
Q_{\beta,t}
:=
\left\{
q\in\mathbb R^{S_t}:
\sum_{j=1}^{S_t}q_j=1,\;
0\leq q_j\leq\frac{p_j^t}{1-\beta},
\; j=1,\ldots,S_t
\right\};
\]
see
\cite{RockafellarUryasev2000,RockafellarUryasev2002}.

\begin{lemma}
\label{lem:bounded_subgradient}

Every sequence \(\{s^k\}\) satisfying
\(
s^k\in\partial H_\rho(w^k)
\)
is bounded. Hence, the associated sequence
\[
u^k
=
s^k+\nu w^k
\in
\partial\widehat H_\rho(w^k)
\]
is also bounded.

\end{lemma}

\begin{proof}

Since \(Q_{\alpha,t}\) is a polytope,~\eqref{eq:cvar_dual} is the
maximum of finitely many affine functions. Hence, by
\cite[Exercise~8.31]{RockafellarWets1998}, for every
\(g^{k,t}\in\partial\psi_{\alpha,t}(w^{k,t})\), there exists
\[
q^{k,t}
\in
\operatorname*{arg\,max}_{q\in Q_{\alpha,t}}
\left\{
-\sum_{j=1}^{S_t}q_j(\xi_j^t)^\top w^{k,t}
\right\}
\]
such that
\(
g^{k,t}
=
-\sum_{j=1}^{S_t}q_j^{k,t}\xi_j^t.
\)
In particular, \(q^{k,t}\in Q_{\alpha,t}\), so
\( 
q_j^{k,t}\geq0,
\;
\sum_{j=1}^{S_t}q_j^{k,t}=1.
\) 
Therefore,
\[
\begin{aligned}
\|g^{k,t}\|
&\leq
\sum_{j=1}^{S_t}q_j^{k,t}\|\xi_j^t\|
\leq
\max_{1\leq j\leq S_t}\|\xi_j^t\|.
\end{aligned}
\]

By Lemma~\ref{lem:subCVaR}, every
\(s^k\in\partial H_\rho(w^k)\) has the form
\[
s^k
=
\rho B_\gamma
\bigl(
g^{k,1},\ldots,g^{k,T}
\bigr),
\;
g^{k,t}\in
\partial\psi_{\alpha,t}(w^{k,t}).
\]
Hence,
\[
\begin{aligned}
\|s^k\|
&=
\rho B_\gamma
\left(
\sum_{t=1}^{T}\|g^{k,t}\|^2
\right)^{1/2}
\leq
\rho B_\gamma
\left[
\sum_{t=1}^{T}
\left(
\max_{1\leq j\leq S_t}\|\xi_j^t\|
\right)^2
\right]^{1/2}.
\end{aligned}
\]
Thus, \(\{s^k\}\) is bounded. Moreover, since \(w^{k,t}\in X\), we have
\(\|w^{k,t}\|\leq1\). Hence,
\[
\|w^k\|^2
=
\sum_{t=1}^{T}\|w^{k,t}\|^2
\leq
T,
\]
and therefore \(\|w^k\|\leq\sqrt T\). Consequently,
\[
\|u^k\|
\leq
\|s^k\|+\nu\|w^k\|
\leq
\|s^k\|+\nu\sqrt T.
\]
Thus, \(\{u^k\}\) is bounded.
\end{proof}

\begin{theorem}
\label{thm:stationary}

Let \(\{x^k\}\) be an infinite sequence generated by
Algorithm~\ref{MainAlgo}. Then every accumulation point of
\(\{x^k\}\) is a critical point of Problem~\eqref{DCP}.

\end{theorem}

\begin{proof}

Since \(\mathcal X\) is compact, the sequence \(\{x^k\}\subset
\mathcal X\) has at least one accumulation point. Let \(\bar x\) be
an arbitrary accumulation point, and choose a subsequence satisfying
\( 
x^{k_\ell}\to\bar x.
\) 
By Proposition~\ref{prop:diff},
\( 
w^{k_\ell}\to\bar x,
\;
y^{k_\ell}\to\bar x.
\) 

The optimality condition of the convex DCA subproblem gives 
\( 
u^k\in\partial\widehat Q_\rho(y^k).
\) 
By Lemma~\ref{lem:bounded_subgradient}, the sequence \(\{u^k\}\) is
bounded. Therefore, after passing to a further subsequence if
necessary, there exists \(\bar u\) such that
\( 
u^{k_\ell}\to\bar u.
\) 

Since \(\widehat Q_\rho\) is proper, lower semicontinuous, and convex,
its subdifferential has a closed graph. Hence,
\[
\bar u
\in
\partial\widehat Q_\rho(\bar x)
=
\partial\widehat G_\rho(\bar x)
+
N_{\mathcal X}(\bar x).
\]
Moreover,
\( 
u^{k_\ell}
\in
\partial\widehat H_\rho(w^{k_\ell}).
\) 
Since \(\widehat H_\rho\) is also proper, lower semicontinuous, and
convex, the closedness of its subdifferential graph gives
\( 
\bar u
\in
\partial\widehat H_\rho(\bar x).
\) 
Therefore,
\[
\bar u
\in
\partial\widehat H_\rho(\bar x)
\cap
\bigl(
\partial\widehat G_\rho(\bar x)
+
N_{\mathcal X}(\bar x)
\bigr).
\]
Thus, \(\bar x\) is critical for the regularized DC decomposition.
By~\eqref{eq:criticality_equivalence}, \(\bar x\) is also critical
for Problem~\eqref{DCP}.

\end{proof}
To prepare the whole-sequence convergence analysis, we incorporate
the constraint \(x\in\mathcal X\) into the objective by defining
\begin{equation}
\label{eq:Psi_def}
\Psi_\rho(x)
:=
\Phi_\rho(x)+\delta_{\mathcal X}(x)
=
\widehat Q_\rho(x)-\widehat H_\rho(x).
\end{equation}
Since \(\Psi_\rho\) may be nonconvex,
\(\partial\Psi_\rho\) denotes its limiting subdifferential
\cite{RockafellarWets1998}.

\medskip
We next recall the Kurdyka--Łojasiewicz (KL) property, which will be
used to establish whole-sequence convergence and a local convergence
rate. A proper closed function \(f\) is said to have the KL property
at \(\bar x\in\operatorname{dom}\partial f\) with exponent
\(\vartheta\in[0,1)\) if there exist \(c>0\), \(\varepsilon>0\), and
\(\zeta>0\) such that
\[
\operatorname{dist}(0,\partial f(x))
\geq
c\bigl(f(x)-f(\bar x)\bigr)^{\vartheta}
\]
whenever
\( 
\|x-\bar x\|<\varepsilon,
\;
f(\bar x)<f(x)<f(\bar x)+\zeta.
\) 
In particular, the exponent \(\vartheta=1/2\) leads to a local linear
convergence rate under the descent and relative-error estimates
established below.

\begin{proposition}
\label{prop:KL}

Under the finite-scenario approximation, \(\Psi_\rho\) is a proper
closed piecewise linear--quadratic function. Moreover, it is
semi-algebraic and has the KL property with exponent \(1/2\) at every
point of \(\operatorname{dom}\partial\Psi_\rho\).

\end{proposition}

\begin{proof}

By Lemma~\ref{lem:VaRDC}, \(R\) is continuous and piecewise affine.
Since the expected-return term is affine and the transaction-cost
terms are piecewise affine, \(\Phi_\rho\) is the sum of a continuous
piecewise-affine function and the quadratic term
\( 
\lambda_2\sum_{t=1}^{T}\|x^t\|^2
=
\lambda_2\|x\|^2.
\) 
Thus, there exists a finite polyhedral partition
\(\{P_r\}_{r=1}^{N}\) of \(\mathcal X\), together with vectors
\(a_r\in\mathbb R^{nT}\) and constants \(b_r\in\mathbb R\), such
that
\[
\Phi_\rho(x)
=
\lambda_2\|x\|^2+a_r^\top x+b_r,
\;
x\in P_r.
\]
Since \(\Psi_\rho=\Phi_\rho\) on \(\mathcal X\) and
\(\Psi_\rho=+\infty\) outside \(\mathcal X\) by~\eqref{eq:Psi_def},
it follows that
\begin{equation}
\label{eq:Psi_PLQ_representation}
\Psi_\rho(x)
=
\min_{1\leq r\leq N}
\left\{
\lambda_2\|x\|^2
+a_r^\top x+b_r
+\delta_{P_r}(x)
\right\}.
\end{equation}

Since each \(P_r\) is a nonempty closed polyhedron, the function
\[
x
\longmapsto
a_r^\top x+b_r+\delta_{P_r}(x)
\]
is proper, closed, and polyhedral. Therefore, \(\Psi_\rho\) is a
proper closed piecewise linear--quadratic function.

Moreover, since \(\Phi_\rho\) is continuous and
\(\Psi_\rho=\Phi_\rho\) on
\(\operatorname{dom}\Psi_\rho=\mathcal X\), the function
\(\Psi_\rho\) is continuous relative to its effective domain and,
hence, on \(\operatorname{dom}\partial\Psi_\rho\). It follows from
\cite[Corollary~5.2]{LiPong2018} that \(\Psi_\rho\) has the KL
property with exponent \(1/2\).

Finally, since \(\delta_{P_r}(x)=0\) for \(x\in P_r\) and
\(+\infty\) otherwise, representation~\eqref{eq:Psi_PLQ_representation}
implies that
\[
\Psi_\rho(x)
=
\lambda_2\|x\|^2+a_r^\top x+b_r,
\;
x\in P_r.
\]
Consequently,
\[
\operatorname{gph}\Psi_\rho
=
\bigcup_{r=1}^{N}
\left\{
(x,\eta):
x\in P_r,\;
\eta=
\lambda_2\|x\|^2+a_r^\top x+b_r
\right\}.
\]
Each set in this finite union is semi-algebraic because \(P_r\) is
polyhedral and the equality defining \(\eta\) is polynomial. Hence,
\(\Psi_\rho\) is semi-algebraic.
\end{proof}

To establish whole-sequence convergence, we next derive a
relative-error estimate under a local no-ties condition. Using the
dual representation~\eqref{eq:cvar_dual}, for
\(t=1,\ldots,T\), define
\[
\mathcal M_{\alpha,t}(z)
:=
\operatorname*{arg\,max}_{q\in Q_{\alpha,t}}
\left\{
-\sum_{j=1}^{S_t}
q_j(\xi_j^t)^\top z
\right\}.
\]

\begin{assumption}[Local no-ties condition]
\label{ass:no_ties}

Let \(\bar x\) be an accumulation point of the sequence generated by
Algorithm~\ref{MainAlgo}. For every \(t=1,\ldots,T\),
\(\mathcal M_{\alpha,t}(\bar x^t)\) is a singleton.

\end{assumption}

\begin{remark}
Assumption~\ref{ass:no_ties} is local and is required only for the
whole-sequence and local linear convergence results. In the
equal-probability case, it holds whenever the scenario losses are
strictly separated across the empirical \(\alpha\)-tail boundary.
For \(S=500\) and \(\alpha=0.95\), this amounts to a strict separation
between the 25th and 26th largest scenario losses. The condition is
generic for continuously distributed scenarios, although ties may
occur under empirical bootstrap sampling.
\end{remark}

\begin{lemma}
\label{lem:relative_error}

Let \(\bar x\) satisfy Assumption~\ref{ass:no_ties}. Then there exists
a neighborhood \(U\) of \(\bar x\) such that, whenever
\(w^k,y^k\in U\),
\[
r^k
:=
u^k-\nabla\widehat H_\rho(y^k)
\in
\partial\Psi_\rho(y^k)
\]
and
\begin{equation}
\label{eq:relative_error}
\operatorname{dist}
\bigl(0,\partial\Psi_\rho(y^k)\bigr)
\leq
\|r^k\|
=
\nu\|d^k\|.
\end{equation}

\end{lemma}

\begin{proof}

For each \(t=1,\ldots,T\), let
\( 
V_{\alpha,t}
:=
\operatorname{ext}(Q_{\alpha,t})
\) 
be the finite set of extreme points of \(Q_{\alpha,t}\). By
\eqref{eq:cvar_dual},
\[
\psi_{\alpha,t}(z)
=
\max_{q\in V_{\alpha,t}}
\left\{
-\sum_{j=1}^{S_t}q_j(\xi_j^t)^\top z
\right\}.
\]

Let \(q_t^*\) be the unique element of
\(\mathcal M_{\alpha,t}(\bar x^t)\), whose existence follows from
Assumption~\ref{ass:no_ties}. Then, for every
\(q\in V_{\alpha,t}\setminus\{q_t^*\}\),
\[
-\sum_{j=1}^{S_t}q_{t,j}^*(\xi_j^t)^\top\bar x^t
>
-\sum_{j=1}^{S_t}q_j(\xi_j^t)^\top\bar x^t.
\]
Since \(V_{\alpha,t}\) is finite, these strict inequalities remain
valid on a neighborhood \(U_t\) of \(\bar x^t\). Hence,
\[
\psi_{\alpha,t}(z)
=
-\sum_{j=1}^{S_t}q_{t,j}^*(\xi_j^t)^\top z,
\;
z\in U_t.
\]
Set \(U:=U_1\times\cdots\times U_T\). By~\eqref{eq:H_def},
\(H_\rho\) is affine on \(U\). Therefore, by
\eqref{eq:regularized_dc}, \(\widehat H_\rho\) is continuously
differentiable on \(U\), and there exists a constant vector \(a\)
such that
\(
\nabla\widehat H_\rho(z)
=
a+\nu z,
\;
z\in U.
\)
Suppose that \(w^k,y^k\in U\). The optimality condition of the convex
DCA subproblem gives
\( 
u^k\in\partial\widehat Q_\rho(y^k).
\)
By~\eqref{eq:Psi_def} and the continuous differentiability of
\(\widehat H_\rho\) on \(U\), the subdifferential sum rule
gives
\[
\partial\Psi_\rho(y^k)
=
\partial\widehat Q_\rho(y^k)
-
\nabla\widehat H_\rho(y^k).
\]
Consequently,
\( 
r^k
:=
u^k-\nabla\widehat H_\rho(y^k)
\in
\partial\Psi_\rho(y^k).
\)  

Moreover, by construction in Algorithm~\ref{MainAlgo},
\( 
u^k\in\partial\widehat H_\rho(w^k).
\) Since \(\widehat H_\rho\) is differentiable on \(U\), this gives
\( 
u^k
=
\nabla\widehat H_\rho(w^k).
\) 
Therefore,
\[
\begin{aligned}
\|r^k\|
&=
\left\|
\nabla\widehat H_\rho(w^k)
-
\nabla\widehat H_\rho(y^k)
\right\|
=
\nu\|w^k-y^k\|
=
\nu\|d^k\|.
\end{aligned}
\]
Finally, since \(r^k\in\partial\Psi_\rho(y^k)\),
\( 
\operatorname{dist}
\bigl(0,\partial\Psi_\rho(y^k)\bigr)
\leq
\|r^k\|
=
\nu\|d^k\|.
\) 
This proves~\eqref{eq:relative_error}.

\end{proof}

We are now ready to state the whole-sequence convergence result.

\begin{theorem}
\label{thm:KL_convergence}

If Algorithm~\ref{MainAlgo} terminates at an iteration \(k\) with
\(d^k=0\), then it returns a critical point of
Problem~\eqref{DCP}. Suppose that the algorithm generates an infinite
sequence and that an accumulation point \(\bar x\) satisfies
Assumption~\ref{ass:no_ties}. Then the sequence has finite length,
\[
\sum_{k=0}^{\infty}\|x^{k+1}-x^k\|<+\infty,
\]
and the whole sequence \(\{x^k\}\) converges to \(\bar x\), which is
a critical point of Problem~\eqref{DCP}.

Moreover, the convergence is \(R\)-linear: there exist constants
\(C>0\), \(\omega\in(0,1)\), and an integer \(K\geq0\) such that
\[
\|x^k-\bar x\|
\leq
C\omega^{k-K},
\;
k\geq K.
\]

\end{theorem}
\begin{proof}

If the algorithm terminates at iteration \(k\) with \(d^k=0\), then
\(w^k=y^k\), and Proposition~\ref{prop:exact_stop} shows that the
returned point is critical for Problem~\eqref{DCP}.

Suppose that the algorithm generates an infinite sequence, and set
\begin{equation}
\label{v_8aug_1}
F_k:=\Phi_\rho(w^k).
\end{equation}
By~\eqref{eq:w_descent}, \(\{F_k\}\) is nonincreasing. Since
\(w^k\in\mathcal X\) and \(\Phi_\rho\) is continuous on the compact
set \(\mathcal X\), the sequence \(\{F_k\}\) is bounded below.
Hence, there exists \(F_*\in\mathbb R\) such that
\begin{equation}
\label{v_8aug_2}
F_k\to F_*.
\end{equation}

Since \(\bar x\) is an accumulation point of \(\{x^k\}\), there
exists a subsequence \(\{x^{k_\ell}\}\) such that
\(x^{k_\ell}\to\bar x\). By Proposition~\ref{prop:diff},
\(w^{k_\ell}\to\bar x\). Therefore, by the continuity of
\(\Phi_\rho\),
\begin{equation}
\label{v_8aug_3}
F_*
=
\lim_{\ell\to\infty}F_{k_\ell}
=
\lim_{\ell\to\infty}\Phi_\rho(w^{k_\ell})
=
\Phi_\rho(\bar x).
\end{equation}

By Theorem~\ref{thm:stationary} and
\eqref{eq:criticality_equivalence}, \(\bar x\) is a critical point
for the regularized DC decomposition. Assumption~\ref{ass:no_ties}
implies that \(\widehat H_\rho\) is continuously differentiable in
a neighborhood of \(\bar x\). Hence, by~\eqref{eq:Psi_def}, the
criticality condition can be written as
\begin{equation}
\label{v_8aug_4}
0
\in
\partial\widehat G_\rho(\bar x)
+
N_{\mathcal X}(\bar x)
-
\nabla\widehat H_\rho(\bar x)
=
\partial\Psi_\rho(\bar x).
\end{equation}
Thus,
\(\bar x\in\operatorname{dom}\partial\Psi_\rho\).
Proposition~\ref{prop:KL} implies that \(\Psi_\rho\) has the KL
property with exponent \(1/2\) at \(\bar x\). Consequently, there
exist \(c>0\), \(\varepsilon>0\), and \(\zeta>0\) such that
\begin{equation}
\label{v_8aug_5}
\operatorname{dist}
\bigl(0,\partial\Psi_\rho(z)\bigr)
\geq
c
\bigl(
\Psi_\rho(z)-\Psi_\rho(\bar x)
\bigr)^{1/2}
\end{equation}
whenever
\begin{equation}
\label{v_8aug_6}
\|z-\bar x\|<\varepsilon,
\;
\Psi_\rho(\bar x)
<
\Psi_\rho(z)
<
\Psi_\rho(\bar x)+\zeta.
\end{equation}

For \(k\geq1\), set
\(a_k:=\|x^k-x^{k-1}\|\). Using
\(x^{k+1}=w^k+\lambda_kd^k\),
\eqref{eq:inertial_difference}, and
\(\lambda_k\leq\bar\lambda\), we obtain
\begin{equation}
\label{v_8aug_7}
\begin{aligned}
a_{k+1}
&=
\|x^{k+1}-x^k\|
\leq
\|w^k-x^k\|
+
\lambda_k\|d^k\|
\leq
\theta a_k
+
\bar\lambda\|d^k\|.
\end{aligned}
\end{equation}

Let \(U\) be the neighborhood given by
Lemma~\ref{lem:relative_error}. Choose \(\delta>0\) such that
\begin{equation}
\label{v_8aug_8}
\overline B(\bar x,3\delta)
\subset
U\cap B(\bar x,\varepsilon),
\;
\chi
:=
\frac{\nu^2}{\nu^2+c^2\sigma}
\in(0,1).
\end{equation}

By Proposition~\ref{prop:diff} and~\eqref{v_8aug_2},
\begin{equation}
\label{v_8aug_9}
F_k-F_*\to0,
\;
\theta a_k\to0,
\;
\|d^k\|\to0.
\end{equation}
Moreover, since \(x^{k_\ell}\to\bar x\), we have
\begin{equation}
\label{v_8aug_10}
\begin{aligned}
&\|x^{k_\ell}-\bar x\|
+
\frac{\theta a_{k_\ell}}{1-\theta}
+
\frac{\bar\lambda}
{(1-\theta)(1-\sqrt\chi)}
\left(
\frac{F_{k_\ell}-F_*}{\sigma}
\right)^{1/2}
\longrightarrow0.
\end{aligned}
\end{equation}

Therefore, by~\eqref{v_8aug_9}--\eqref{v_8aug_10}, there exists
\(\ell_0\) sufficiently large such that, with
\(K:=k_{\ell_0}\geq1\),
\begin{equation}
\label{v_8aug_11}
F_K-F_*<\zeta,
\;
\theta a_k<\delta,
\;
\|d^k\|<\delta,
\; k\geq K,
\end{equation}
and
\begin{equation}
\label{v_8aug_12}
\|x^K-\bar x\|
+
\frac{\theta a_K}{1-\theta}
+
\frac{\bar\lambda}
{(1-\theta)(1-\sqrt\chi)}
\left(
\frac{F_K-F_*}{\sigma}
\right)^{1/2}
<
\delta.
\end{equation}
Set
\begin{equation}
\label{v_8aug_13}
C_1
:=
\left(
\frac{F_K-F_*}{\sigma}
\right)^{1/2}.
\end{equation}

We prove by induction that, for every \(n\geq K\),
\begin{equation}
\label{v_8aug_14}
x^n\in B(\bar x,\delta),
\;
F_n-F_*
\leq
\chi^{n-K}(F_K-F_*),
\end{equation}
and
\begin{equation}
\label{v_8aug_15}
\|d^n\|
\leq
C_1\chi^{(n-K)/2}.
\end{equation}

For \(n=K\), the first assertion in~\eqref{v_8aug_14} follows
from~\eqref{v_8aug_12}, while
\[
F_K-F_*
=
\chi^{K-K}(F_K-F_*).
\]
Thus,~\eqref{v_8aug_14} holds at \(n=K\). Moreover,
\eqref{eq:w_descent}, \(\lambda_K\geq1\), and \(F_{K+1}\geq F_*\)
give
\[
\sigma\|d^K\|^2
\leq
F_K-F_{K+1}
\leq
F_K-F_*.
\]
Hence,
\[
\|d^K\|
\leq
\left(
\frac{F_K-F_*}{\sigma}
\right)^{1/2}
=
C_1
=
C_1\chi^{(K-K)/2},
\]
so~\eqref{v_8aug_15} also holds at \(n=K\).

Now suppose that, for some \(q\geq K\),
\eqref{v_8aug_14}--\eqref{v_8aug_15} hold for every
\(j=K,\ldots,q\). We prove that both estimates hold at \(q+1\).

By the triangle inequality, \eqref{eq:inertial_difference},
\eqref{v_8aug_11}, and the induction hypothesis,
\begin{equation}
\label{v_8aug_16}
\begin{aligned}
\|w^q-\bar x\|
&
\leq
\theta a_q+\|x^q-\bar x\|
<
\delta+\delta
=
2\delta,
\\
\|y^q-\bar x\|
&
\leq
\|y^q-w^q\|+\|w^q-\bar x\|
=
\|d^q\|+\|w^q-\bar x\|
<
\delta+2\delta
=
3\delta.
\end{aligned}
\end{equation}
It follows from~\eqref{v_8aug_8} and~\eqref{v_8aug_16} that
\( 
w^q,y^q\in U,
\;
y^q\in B(\bar x,\varepsilon).
\) 

Furthermore, by~\eqref{eq:bracketing} and~\eqref{v_8aug_11},
\begin{equation}
\label{v_8aug_17}
F_*
\leq
F_{q+1}
\leq
\Phi_\rho(y^q)
\leq
F_q
\leq
F_K
<
F_*+\zeta.
\end{equation}

By the definition of the convex subproblem,
\(y^q\in\mathcal X\). Moreover,
\(x^{k_\ell}\in\mathcal X\),
\(x^{k_\ell}\to\bar x\), and \(\mathcal X\) is closed; hence,
\(\bar x\in\mathcal X\). Therefore, by~\eqref{eq:Psi_def} and
\eqref{v_8aug_3},
\begin{equation}
\label{v_8aug_18}
\Psi_\rho(y^q)=\Phi_\rho(y^q),
\;
\Psi_\rho(\bar x)=F_*.
\end{equation}

We distinguish two cases. Suppose first that
\(\Phi_\rho(y^q)>F_*\). By~\eqref{v_8aug_17} and
\eqref{v_8aug_18},
\[
\Psi_\rho(\bar x)
<
\Psi_\rho(y^q)
<
\Psi_\rho(\bar x)+\zeta.
\]
Together with \(y^q\in B(\bar x,\varepsilon)\), this verifies
the conditions in~\eqref{v_8aug_6} with \(z=y^q\). Hence,
the KL inequality~\eqref{v_8aug_5} applies at \(y^q\). Moreover, since \(w^q,y^q\in U\),
Lemma~\ref{lem:relative_error} gives
\[
\operatorname{dist}
\bigl(0,\partial\Psi_\rho(y^q)\bigr)
\leq
\nu\|d^q\|.
\]
Combining this inequality with 
~\eqref{v_8aug_5} at \(z=y^q\), gives
\[
c\bigl(\Phi_\rho(y^q)-F_*\bigr)^{1/2}
\leq
\operatorname{dist}
\bigl(0,\partial\Psi_\rho(y^q)\bigr)
\leq
\nu\|d^q\|.
\]
and therefore
\[
\Phi_\rho(y^q)-F_*
\leq
\frac{\nu^2}{c^2}\|d^q\|^2.
\]
Using~\eqref{v_8aug_17}, we obtain
\begin{equation}
\label{v_8aug_19}
F_{q+1}-F_*
\leq
\Phi_\rho(y^q)-F_*
\leq
\frac{\nu^2}{c^2}\|d^q\|^2.
\end{equation}

Suppose now that
\(\Phi_\rho(y^q)=F_*\). Then~\eqref{v_8aug_17} gives
\[
F_*
\leq
F_{q+1}
\leq
\Phi_\rho(y^q)
=
F_*,
\]
so \(F_{q+1}=F_*\). Therefore,
\[
F_{q+1}-F_*
=
\Phi_\rho(y^q)-F_*
=
0,
\]
and~\eqref{v_8aug_19} also holds in this case. Thus,~\eqref{v_8aug_19} holds in both cases. On the other hand,
\eqref{eq:w_descent} and \(\lambda_q\geq1\) give
\begin{equation}
\label{v_8aug_20}
F_q-F_{q+1}
\geq
\sigma\lambda_q^2\|d^q\|^2
\geq
\sigma\|d^q\|^2.
\end{equation}
By~\eqref{v_8aug_19},
\[
\|d^q\|^2
\geq
\frac{c^2}{\nu^2}
\bigl(F_{q+1}-F_*\bigr).
\]
Combining this inequality with~\eqref{v_8aug_20}, we obtain
\[
F_q-F_{q+1}
\geq
\frac{c^2\sigma}{\nu^2}
\bigl(F_{q+1}-F_*\bigr).
\]
Therefore,
\[
\begin{aligned}
F_q-F_*
&=
F_q-F_{q+1}
+
F_{q+1}-F_*
\geq
\left(
1+\frac{c^2\sigma}{\nu^2}
\right)
\bigl(F_{q+1}-F_*\bigr).
\end{aligned}
\]
Using the definition of \(\chi\) in~\eqref{v_8aug_8} and the
induction hypothesis~\eqref{v_8aug_14}, we conclude that
\begin{equation}
\label{v_8aug_21}
F_{q+1}-F_*
\leq
\chi(F_q-F_*)
\leq
\chi^{q+1-K}(F_K-F_*).
\end{equation}
Thus,~\eqref{v_8aug_21} proves the second assertion in
\eqref{v_8aug_14} at \(q+1\).

It remains to prove that \(x^{q+1}\in B(\bar x,\delta)\).
Summing~\eqref{v_8aug_7} from \(j=K\) to \(q\) gives
\[
\sum_{j=K+1}^{q+1}a_j
\leq
\theta a_K
+
\theta\sum_{j=K+1}^{q}a_j
+
\bar\lambda\sum_{j=K}^{q}\|d^j\|.
\]
Since \(a_j\geq0\),
\[
\sum_{j=K+1}^{q}a_j
\leq
\sum_{j=K+1}^{q+1}a_j.
\]
Therefore,
\[
(1-\theta)
\sum_{j=K+1}^{q+1}a_j
\leq
\theta a_K
+
\bar\lambda\sum_{j=K}^{q}\|d^j\|.
\]
By the induction hypothesis~\eqref{v_8aug_15},
\[
\begin{aligned}
\sum_{j=K}^{q}\|d^j\|
&\leq
C_1\sum_{j=K}^{q}(\sqrt\chi)^{j-K}
\leq
\frac{C_1}{1-\sqrt\chi}.
\end{aligned}
\]
Hence,
\begin{equation}
\label{v_8aug_22}
\sum_{j=K+1}^{q+1}a_j
\leq
\frac{\theta a_K}{1-\theta}
+
\frac{\bar\lambda C_1}
{(1-\theta)(1-\sqrt\chi)}.
\end{equation}
Consequently, by~\eqref{v_8aug_12}, \eqref{v_8aug_13}, and
\eqref{v_8aug_22},
\[
\begin{aligned}
\|x^{q+1}-\bar x\|
&=
\left\|
x^K-\bar x
+
\sum_{j=K+1}^{q+1}(x^j-x^{j-1})
\right\|
\leq
\|x^K-\bar x\|
+
\sum_{j=K+1}^{q+1}a_j
<
\delta.
\end{aligned}
\]
Together with~\eqref{v_8aug_21}, this proves both assertions in
\eqref{v_8aug_14} at \(q+1\). Finally, by~\eqref{eq:w_descent},
\(\lambda_{q+1}\geq1\), \(F_{q+2}\geq F_*\), and
\eqref{v_8aug_21},
\[
\begin{aligned}
\|d^{q+1}\|^2
&\leq
\frac{F_{q+1}-F_{q+2}}{\sigma}
\leq
\frac{F_{q+1}-F_*}{\sigma}
\leq
\frac{\chi^{q+1-K}(F_K-F_*)}{\sigma}
=
C_1^2\chi^{q+1-K},
\end{aligned}
\]
where~\eqref{v_8aug_13} was used in the last equality. Hence,
\[
\|d^{q+1}\|
\leq
C_1\chi^{(q+1-K)/2},
\]
which proves~\eqref{v_8aug_15} at \(q+1\). Thus,~\eqref{v_8aug_14}--\eqref{v_8aug_15} hold at \(q+1\).
This completes the induction step. Therefore, by induction,
\eqref{v_8aug_14}--\eqref{v_8aug_15} hold for every \(n\geq K\).

Letting \(q\to\infty\) in~\eqref{v_8aug_22} gives
\( 
\sum_{j=K+1}^{\infty}a_j
<
+\infty.
\) 
Adding the finite initial part, we obtain
\[
\begin{aligned}
\sum_{k=0}^{\infty}\|x^{k+1}-x^k\|
&=
\sum_{j=1}^{\infty}a_j
=
\sum_{j=1}^{K}a_j
+
\sum_{j=K+1}^{\infty}a_j
<
+\infty.
\end{aligned}
\]
Thus, \(\{x^k\}\) has finite length and is a Cauchy sequence, so it
converges. Since \(\bar x\) is an accumulation point, the whole
sequence converges to \(\bar x\), which is critical by
Theorem~\ref{thm:stationary}.

It remains to establish the convergence rate. Combining
\eqref{v_8aug_7} and~\eqref{v_8aug_15} gives
\begin{equation}
\label{v_8aug_23}
a_{k+1}
\leq
\theta a_k
+
\bar\lambda C_1(\sqrt\chi)^{k-K},
\;
k\geq K.
\end{equation}
Set
\( 
\omega_0:=\max\{\theta,\sqrt\chi\}
\;\text{and choose}\;
\omega\in(\omega_0,1).
\) 
Iterating~\eqref{v_8aug_23}, for every \(r\geq1\), gives
\begin{equation}
\label{v_8aug_24}
\begin{aligned}
a_{K+r}
&\leq
\theta^r a_K
+
\bar\lambda C_1
\sum_{j=0}^{r-1}
\theta^{r-1-j}(\sqrt\chi)^j
\leq
a_K\omega^r
+
\bar\lambda C_1r\omega_0^{r-1}.
\end{aligned}
\end{equation}
Since \(\omega_0/\omega<1\), the constant
\( 
D_0
:=
\sup_{r\geq1}
r\left(\frac{\omega_0}{\omega}\right)^{r-1}
\) 
is finite. Therefore,~\eqref{v_8aug_24} gives
\[
a_{K+r}
\leq
\left(
a_K+\frac{\bar\lambda C_1D_0}{\omega}
\right)
\omega^r.
\]
Hence, there exists \(C_2>0\) such that
\begin{equation}
\label{v_8aug_25}
a_k
\leq
C_2\omega^{k-K},
\;
k\geq K.
\end{equation}

Finally, since \(x^k\to\bar x\), the telescoping identity and
\eqref{v_8aug_25} give
\[
\begin{aligned}
\|x^k-\bar x\|
&=
\left\|
\sum_{j=k+1}^{\infty}(x^{j-1}-x^j)
\right\|
\leq
\sum_{j=k+1}^{\infty}a_j
\leq
C_2\sum_{j=k+1}^{\infty}\omega^{j-K}
=
\frac{C_2\omega}{1-\omega}
\omega^{k-K}.
\end{aligned}
\]
Setting
\( 
C:=\frac{C_2\omega}{1-\omega}
\) 
yields
\(
\|x^k-\bar x\|
\leq
C\omega^{k-K},
\;
k\geq K.
\)
Thus, the convergence is \(R\)-linear.

\end{proof}
\section{Numerical Results}
\label{sec5}

This section evaluates the computational performance of the proposed
projected inertial BDCA (iBDCA) and the out-of-sample performance of
the resulting portfolio strategies. We first compare iBDCA with DCA
and BDCA on matched rolling instances, then examine the robustness of
the three methods with respect to the VaR threshold, and finally
evaluate the resulting portfolio strategies against the equal-weight
(EW) and buy-and-hold (BH) benchmarks.

\subsection{Experimental Setup}
\label{subsec:experimental_setup}

We consider two datasets, each containing adjusted daily prices for
\(n=30\) stocks selected from the NASDAQ and S\&P 500, respectively,
over the period from August 8, 2022 to August 5, 2026. The lists of assets are reported in
Appendix~\ref{tab:asset_lists}.

At each rebalancing date, \(S=500\) equally weighted scenario paths
of length \(T=10\) are generated from the most recent 252 returns
using a moving-block bootstrap with block size 5. The multi-period
problem is solved, and only its first-period allocation is implemented.
Let \(N_{\mathrm{OS}}\) denote the number of out-of-sample rebalancing
periods; in our experiments, \(N_{\mathrm{OS}}=749\).

For each dataset, we perform \(N_{\mathrm{run}}=60\) independent
replications under two initialization schemes. The seed for replication
\(r\) and scheme \(h\in\{1,2\}\) is
\(\texttt{2026+2*(r-1)+h}\). Initial points are generated by
normalizing independent random variables
\(z_i\sim\operatorname{Gamma}(a_h,1)\), with \(a_1=100\) and
\(a_2=0.5\), producing nearly uniform and more dispersed initial
points, respectively.

The scenario probabilities are \(p_j=1/S\), \(j=1,\ldots,S\).
Unless otherwise stated, we use \(\lambda_1=1\),
\(\lambda_2=0.002\), \(\rho=5\), \(\alpha=0.95\),
\(\tau=0.02\), \(\gamma=0.5/S=0.001\), and
\(c_i=5\times10^{-4}\). For \(S=500\),
equation~\eqref{v_5aug_1} gives \(\bar\varepsilon=0.002\), so
\(0<\gamma<\bar\varepsilon\), as required in
Lemma~\ref{lem:VaRDC}. 

For iBDCA, we use \(\theta=0.01\), \(\nu=2\times10^{-3}\),
\(\sigma=5\times10^{-4}\), \(\eta=0.5\),
\(\bar\lambda=1.05\), \(K_{\max}=50\), and
\(\varepsilon_{\mathrm{tol}}=10^{-5}\). Preliminary tests indicated no
consistent improvement for larger values of \(\theta\), so we use
\(\theta=0.01\) throughout. BDCA uses the same parameter values with
\(\theta=0\). The stopping test is
\(\|d^k\|\leq\varepsilon_{\mathrm{tol}}\max\{1,\|w^k\|\}\), and the
VaR feasibility tolerance is \(10^{-6}\).

All experiments are implemented in MATLAB R2026a using CVX~2.2 and
MOSEK~11.2.2 and performed on a Mac mini M4 Pro with a 12-core CPU
and 48 GB unified memory.

\subsection{Matched Solver Comparison}
\label{subsec:matched_solver}

For each matched instance \((r,\ell)\), let \(x^{(r,\ell)}\) denote
the BDCA or iBDCA solution and \(x_{\mathrm{DCA}}^{(r,\ell)}\) the
corresponding DCA solution. The paired penalized-objective difference
relative to DCA is defined by
\(
\Delta\Phi^{(r,\ell)}
:=
\Phi_\rho\bigl(x^{(r,\ell)}\bigr)
-
\Phi_\rho\bigl(x_{\mathrm{DCA}}^{(r,\ell)}\bigr).
\)
Constraint satisfaction is assessed through the aggregate VaR
violation \(R(x)\) and the maximum periodwise violation
\[
V(x)
:=
\max_{t=1,\ldots,T}
\left(
\operatorname{VaR}_{\alpha}
\left(
-(\xi^t)^\top x^t
\right)
-\tau
\right)_+.
\]
A solution is regarded as scenario-feasible if \(V(x)\leq10^{-6}\).

For each replication, the objective difference, aggregate violation,
iteration count, and solution time are averaged over the
\(N_{\mathrm{OS}}\) rolling instances. The win and scenario-feasible
rates are the percentages of instances with
\(\Delta\Phi^{(r,\ell)}<0\) and scenario-feasible solutions,
respectively. Table~\ref{tab:solver_results} reports the mean and
standard deviation of these run-level quantities over the independent
replications.

\begin{table}[H]
\centering
\footnotesize
\setlength{\tabcolsep}{2pt}
\resizebox{\linewidth}{!}{%
\begin{tabular}{cclcccccc}
\hline
Dataset
& Scheme
& Method
& \shortstack{Mean \(\Delta\Phi\)\\\(\downarrow\)}
& \shortstack{Win (\%)\\\(\uparrow\)}
& \shortstack{Mean \(R\)\\\(\downarrow\)}
& \shortstack{Feas. (\%)\\\(\uparrow\)}
& \shortstack{Iter.\\\(\downarrow\)}
& \shortstack{Time (s)\\\(\downarrow\)} \\
\hline
\textbf{NASDAQ}
& 1 & DCA
& \(0\)
& --
& \((\mathbf{7.6418}\pm2.3315)\times10^{-18}\)
& \(\mathbf{100.00}\pm0.00\)
& \(\mathbf{8.21}\pm0.04\)
& \(\mathbf{10.9431}\pm0.9855\) \\
& 1 & BDCA
& \((-2.1453\pm0.0747)\times10^{-4}\)
& \(88.31\pm1.40\)
& \((9.9283\pm34.2960)\times10^{-15}\)
& \(\mathbf{100.00}\pm0.00\)
& \(9.99\pm0.13\)
& \(13.3911\pm1.2711\) \\
& 1 & iBDCA
& \((\mathbf{-1.6610}\pm0.0339)\times10^{-3}\)
& \(\mathbf{93.51}\pm0.85\)
& \((1.9456\pm4.2814)\times10^{-14}\)
& \(\mathbf{100.00}\pm0.00\)
& \(11.45\pm0.16\)
& \(15.6045\pm1.4744\) \\
\hline
& 2 & DCA
& \(0\)
& --
& \((2.1837\pm1.4014)\times10^{-5}\)
& \(99.28\pm0.35\)
& \(\mathbf{8.32}\pm0.06\)
& \(\mathbf{10.2419}\pm0.7217\) \\
& 2 & BDCA
& \((-3.6374\pm0.2978)\times10^{-4}\)
& \(89.19\pm0.83\)
& \((1.7596\pm1.2181)\times10^{-5}\)
& \(99.35\pm0.33\)
& \(10.21\pm0.08\)
& \(12.6528\pm0.8957\) \\
& 2 & iBDCA
& \((\mathbf{-3.0923}\pm0.1577)\times10^{-3}\)
& \(\mathbf{94.66}\pm0.73\)
& \((\mathbf{3.9239}\pm3.9897)\times10^{-6}\)
& \(\mathbf{99.71}\pm0.26\)
& \(11.93\pm0.19\)
& \(15.0729\pm0.9885\) \\
\hline\hline
\textbf{S\&P 500}
& 1 & DCA
& \(0\)
& --
& \(\mathbf{0}\pm0\)
& \(\mathbf{100.00}\pm0.00\)
& \(\mathbf{8.93}\pm0.04\)
& \(\mathbf{11.2420}\pm0.6502\) \\
& 1 & BDCA
& \((-1.5775\pm0.0818)\times10^{-4}\)
& \(90.84\pm1.19\)
& \(\mathbf{0}\pm0\)
& \(\mathbf{100.00}\pm0.00\)
& \(10.75\pm0.16\)
& \(13.6297\pm0.8910\) \\
& 1 & iBDCA
& \((\mathbf{-5.1354}\pm0.0389)\times10^{-3}\)
& \(\mathbf{100.00}\pm0.00\)
& \((1.2137\pm2.5985)\times10^{-14}\)
& \(\mathbf{100.00}\pm0.00\)
& \(16.13\pm0.18\)
& \(21.3088\pm1.3233\) \\
\hline
& 2 & DCA
& \(0\)
& --
& \(\mathbf{0}\pm0\)
& \(\mathbf{100.00}\pm0.00\)
& \(\mathbf{9.04}\pm0.06\)
& \(\mathbf{10.8654}\pm0.7081\) \\
& 2 & BDCA
& \((-2.8604\pm0.1452)\times10^{-4}\)
& \(93.18\pm0.90\)
& \(\mathbf{0}\pm0\)
& \(\mathbf{100.00}\pm0.00\)
& \(11.16\pm0.14\)
& \(13.4934\pm0.9111\) \\
& 2 & iBDCA
& \((\mathbf{-6.5920}\pm0.1048)\times10^{-3}\)
& \(\mathbf{99.98}\pm0.05\)
& \((1.0893\pm2.8465)\times10^{-15}\)
& \(\mathbf{100.00}\pm0.00\)
& \(16.44\pm0.20\)
& \(20.7334\pm1.3499\) \\
\hline
\end{tabular}%
}
\caption{Matched solver comparison under the two
initialization schemes.}
\label{tab:solver_results}
\end{table}

Figure~\ref{fig:objective_differences} shows the mean paired
differences in the final penalized objective values at each rolling
period, with pointwise \(95\%\) confidence intervals across the
independent replications. Figure~\ref{fig:convergence} shows the
within-instance evolution of \(\Phi_\rho\) and the relative residual
\(\|d^k\|/\max\{1,\|w^k\|\}\). For each dataset and initialization
scheme, the representative matched instance is selected as the one
whose final DCA objective is closest to the median over all
replications and rolling periods.

\begin{figure}[H]
\centering
\begin{subfigure}{0.48\textwidth}
\centering
\includegraphics[width=\linewidth]
{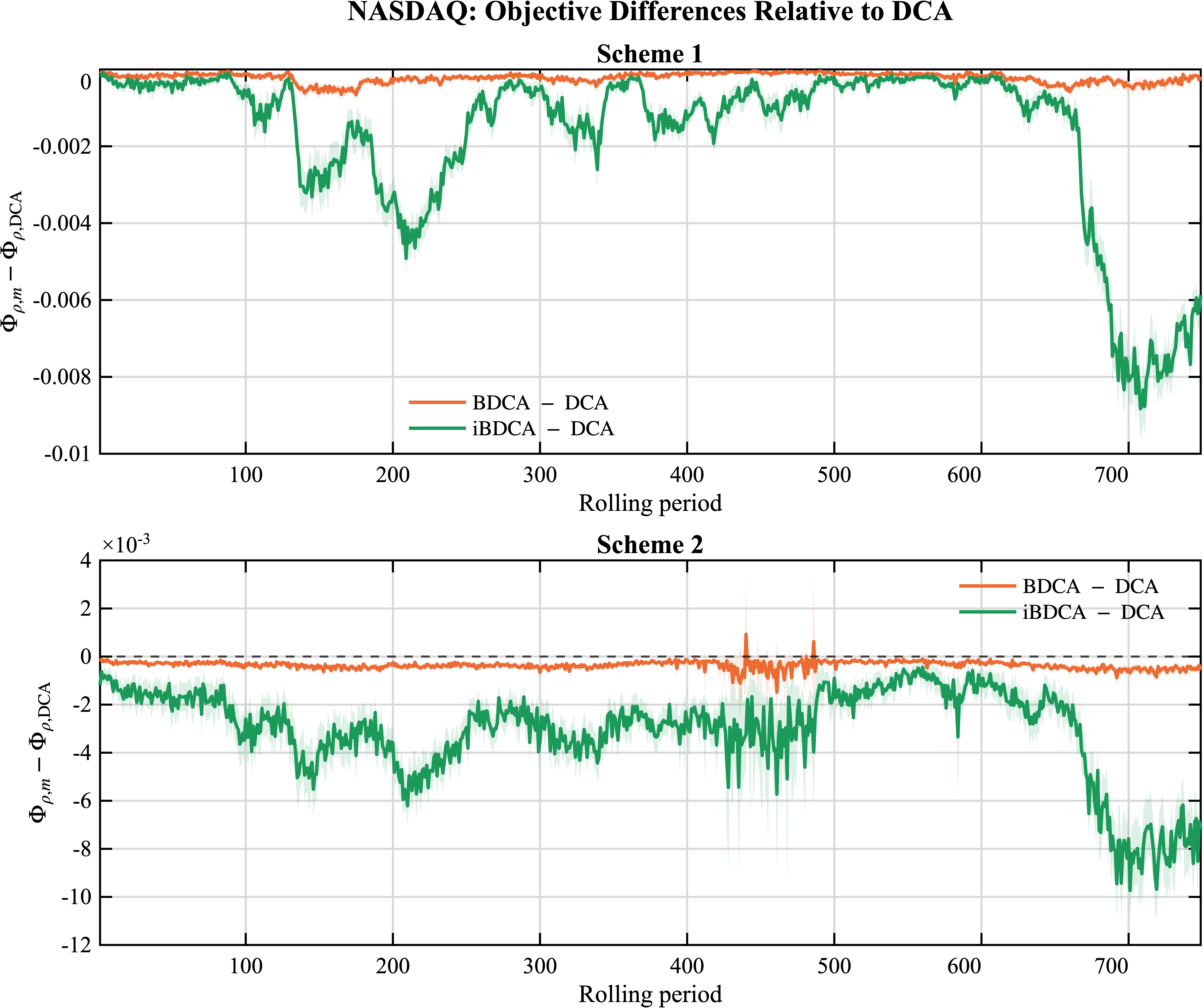}
\caption{NASDAQ}
\label{fig:objective_differences_nasdaq}
\end{subfigure}
\hfill
\begin{subfigure}{0.48\textwidth}
\centering
\includegraphics[width=\linewidth]
{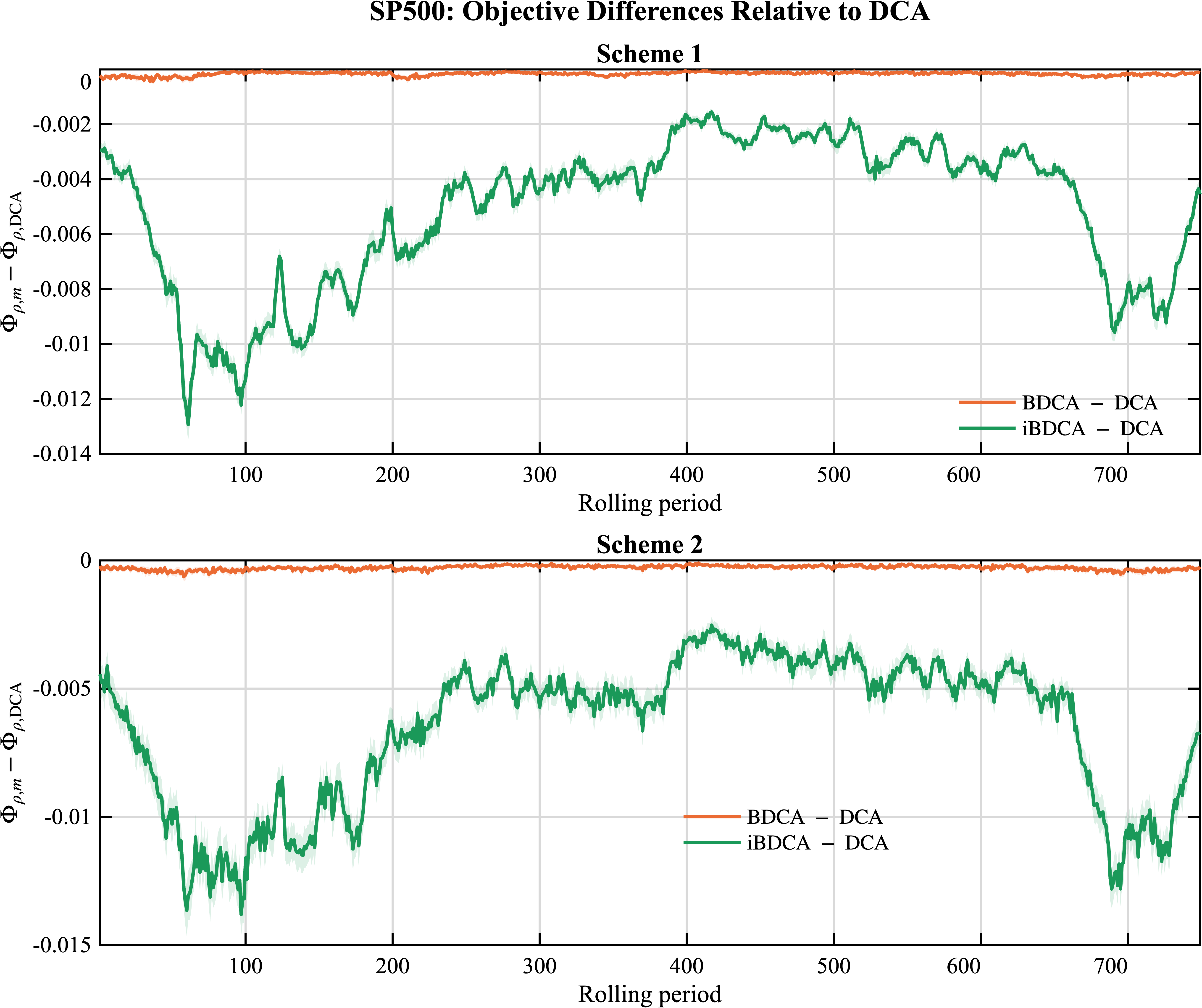}
\caption{S\&P 500}
\label{fig:objective_differences_sp500}
\end{subfigure}
\caption{Objective differences relative to DCA.}
\label{fig:objective_differences}
\end{figure}

\begin{figure}[H]
\centering
\begin{subfigure}{0.48\textwidth}
\centering
\includegraphics[width=\linewidth]
{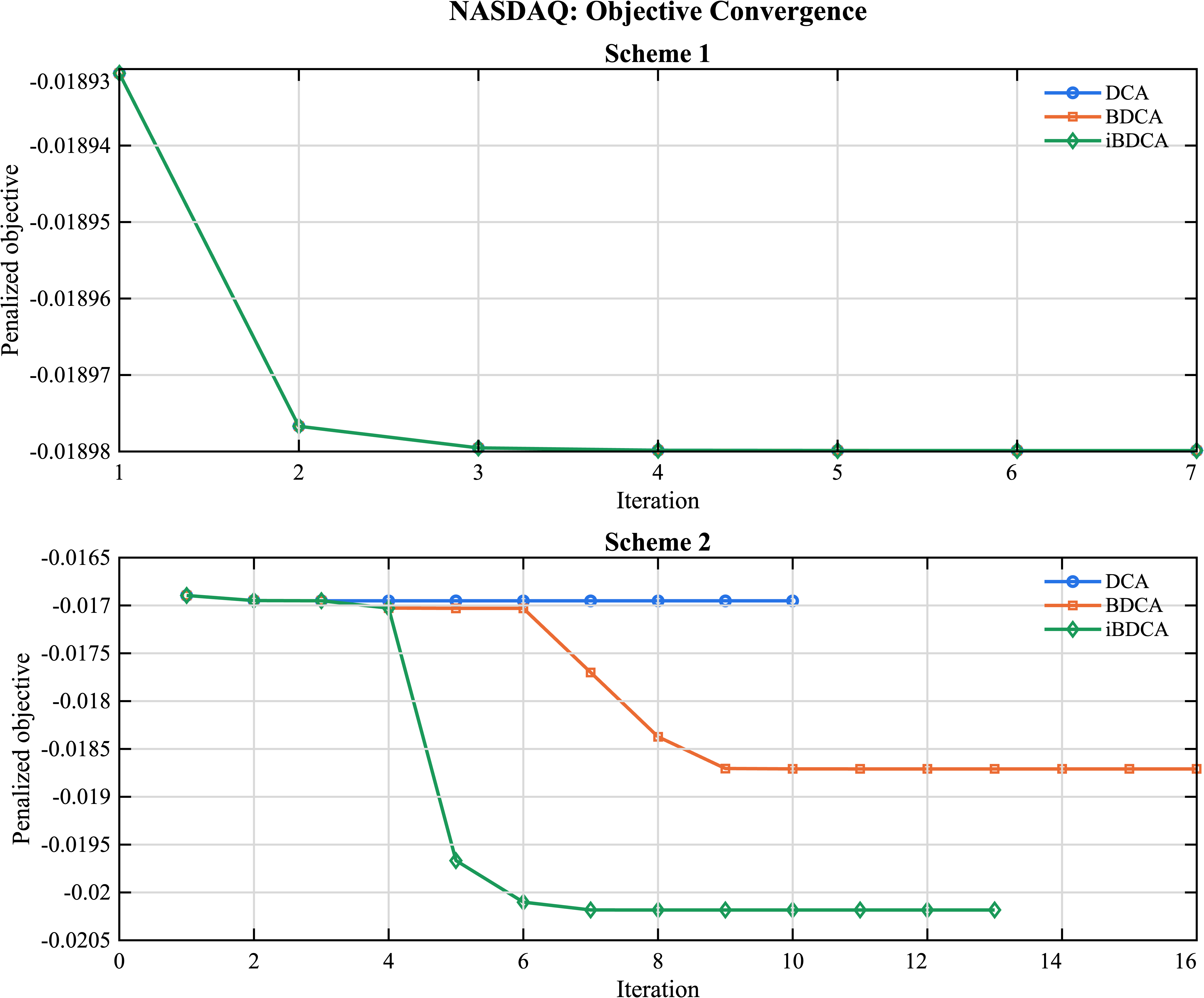}
\caption{NASDAQ: objective}
\label{fig:objective_convergence_nasdaq}
\end{subfigure}
\hfill
\begin{subfigure}{0.48\textwidth}
\centering
\includegraphics[width=\linewidth]
{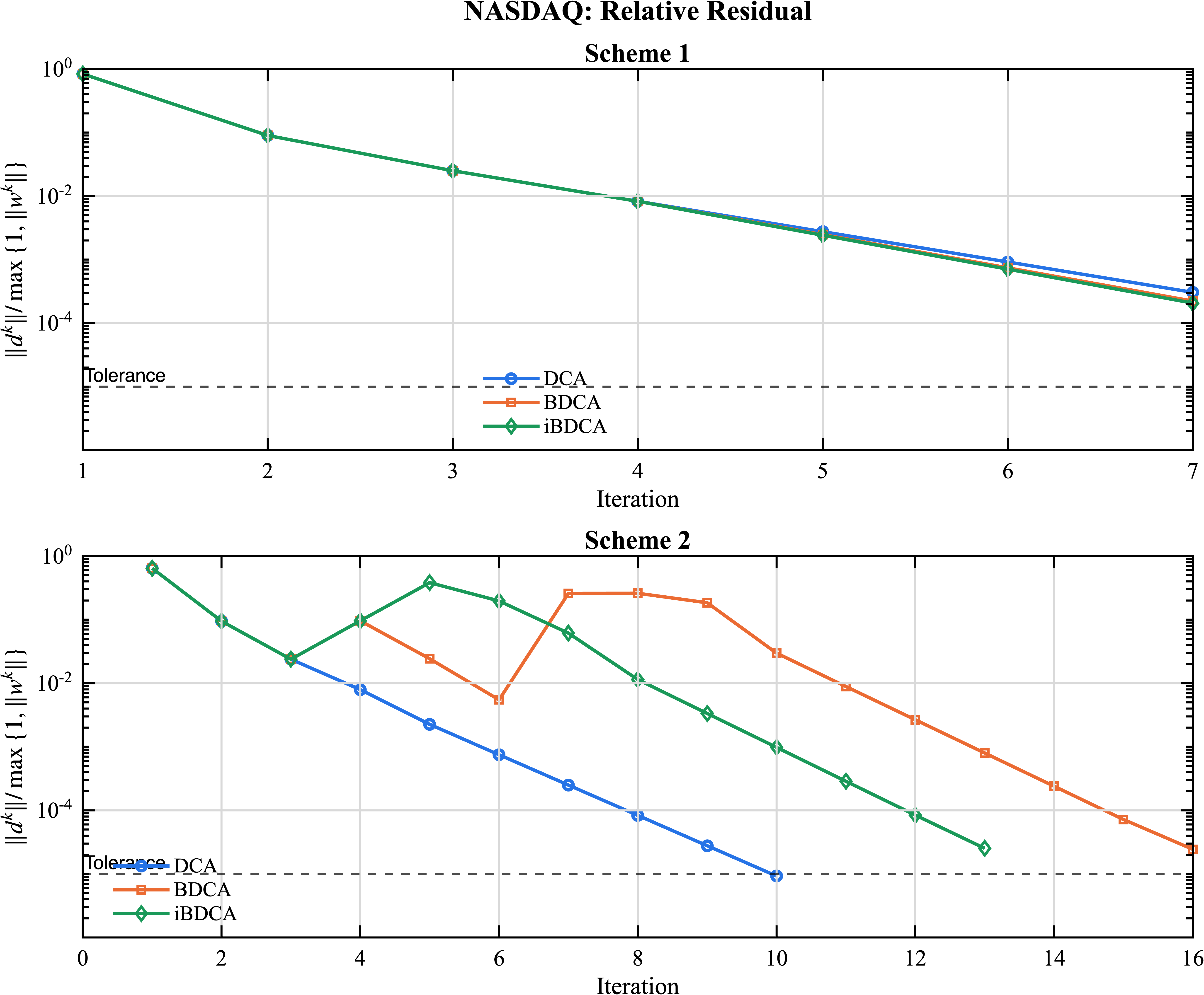}
\caption{NASDAQ: residual}
\label{fig:residual_convergence_nasdaq}
\end{subfigure}

\medskip

\begin{subfigure}{0.48\textwidth}
\centering
\includegraphics[width=\linewidth]
{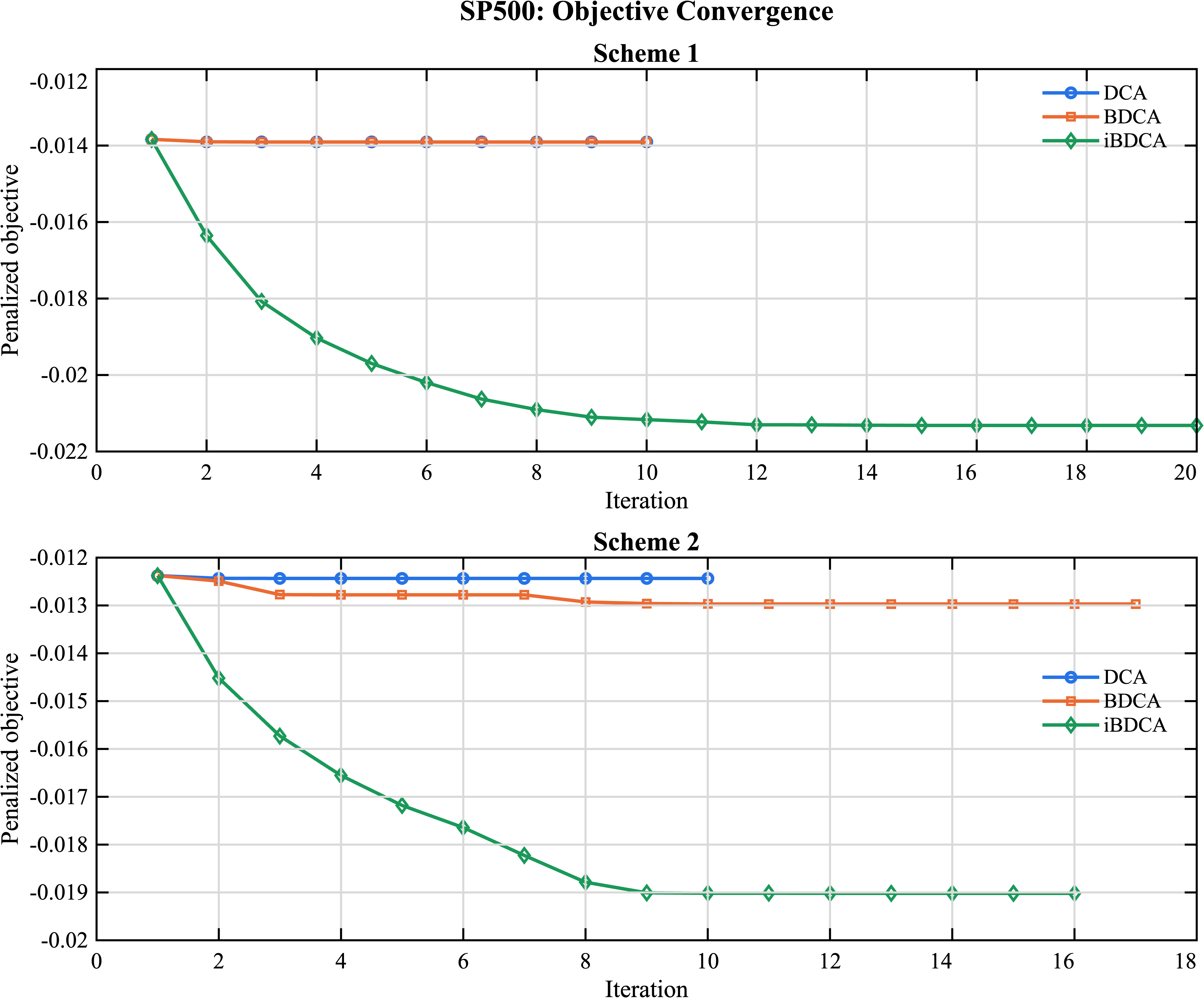}
\caption{S\&P 500: objective}
\label{fig:objective_convergence_sp500}
\end{subfigure}
\hfill
\begin{subfigure}{0.48\textwidth}
\centering
\includegraphics[width=\linewidth]
{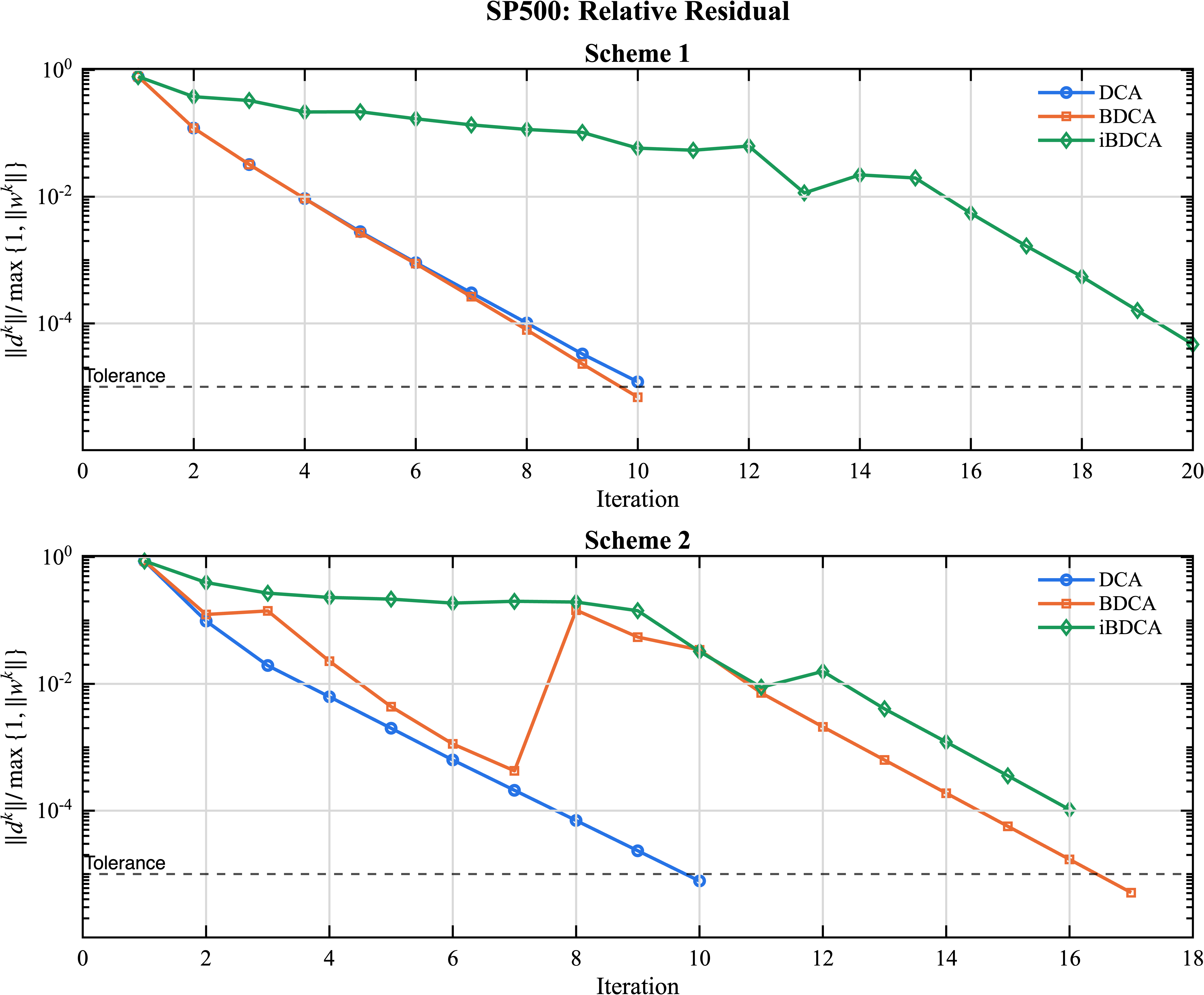}
\caption{S\&P 500: residual}
\label{fig:residual_convergence_sp500}
\end{subfigure}
\caption{Convergence behavior of DCA, BDCA, and iBDCA.}
\label{fig:convergence}
\end{figure}
\begin{remark}
The matched solver results reveal a clear trade-off between attained
penalized-objective value and computational cost. Since the problem is
nonconvex, the methods may converge to different critical points. On
both datasets, BDCA attains moderately lower objective values than DCA,
whereas iBDCA achieves substantially lower objective values and the
highest win rates, at the expense of more iterations and longer solution
times. For NASDAQ under Scheme~1, all three methods are fully
scenario feasible and the aggregate VaR violations are essentially at
numerical precision. Under Scheme~2, iBDCA yields the smallest mean
aggregate VaR violation and the highest scenario-feasible rate. For the
S\&P 500 dataset, all three methods attain a scenario-feasible rate of
\(100\%\) under both initialization schemes, while the corresponding
aggregate VaR violations are zero or within numerical
precision. Thus,
iBDCA attains better penalized-objective values while maintaining a high
level of scenario-based VaR feasibility, although this improvement comes
with higher computational cost.
\end{remark}
\subsection{Robustness with Respect to the VaR Threshold}
\label{subsec:robustness}

To isolate the effect of the VaR threshold, we fix Scheme~2 and use the
same matched rolling instances and evaluation metrics as in
Subsection~\ref{subsec:matched_solver}. We assess the robustness of the
three methods for
\(\tau\in\{1.5,\;1.75,\;2\}\times10^{-2}\), while keeping all other
parameters fixed. Guided by the empirical \(95\%\) VaR levels of the
individual assets in the two datasets, these values are chosen to
represent increasingly less restrictive risk requirements, from the
most stringent case \(\tau=0.015\) to the reference case
\(\tau=0.02\). For each \(\tau\), objective differences and win rates
for BDCA and iBDCA are computed relative to DCA.
\begin{table}[H]
\centering
\footnotesize
\setlength{\tabcolsep}{2pt}
\renewcommand{\arraystretch}{0.88}

\resizebox{\linewidth}{!}{%
\begin{tabular}{cllcccccc}
\hline
Dataset
& \(\tau\;(\times10^{-2})\)
& Method
& \shortstack{Mean \(\Delta\Phi\)\\\(\downarrow\)}
& \shortstack{Win (\%)\\\(\uparrow\)}
& \shortstack{Mean \(R\)\\\(\downarrow\)}
& \shortstack{Feas. (\%)\\\(\uparrow\)}
& \shortstack{Mean Iter.\\\(\downarrow\)}
& \shortstack{Time (s)\\\(\downarrow\)} \\
\hline
\textbf{NASDAQ}
& \(1.5\)
& DCA
& \(0\)
& --
& \((3.4391\pm0.8821)\times10^{-4}\)
& \(92.70\pm0.83\)
& \(\mathbf{7.77}\pm0.07\)
& \(\mathbf{10.3494}\pm0.6191\) \\
&
& BDCA
& \((-4.6135\pm1.3292)\times10^{-4}\)
& \(85.76\pm1.32\)
& \((2.8895\pm0.7750)\times10^{-4}\)
& \(93.47\pm0.62\)
& \(9.01\pm0.15\)
& \(12.0547\pm0.7175\) \\
&
& iBDCA
& \((\mathbf{-1.8566}\pm0.3875)\times10^{-3}\)
& \(\mathbf{88.12}\pm1.03\)
& \((\mathbf{1.3277}\pm0.4386)\times10^{-4}\)
& \(\mathbf{95.39}\pm0.69\)
& \(9.87\pm0.16\)
& \(13.2474\pm0.7221\) \\
\cline{2-9}
& \(1.75\)
& DCA
& \(0\)
& --
& \((9.3627\pm4.4647)\times10^{-5}\)
& \(98.30\pm0.50\)
& \(\mathbf{8.03}\pm0.05\)
& \(\mathbf{10.4549}\pm0.7436\) \\
&
& BDCA
& \((-3.2943\pm0.5926)\times10^{-4}\)
& \(87.45\pm1.07\)
& \((7.8438\pm4.3138)\times10^{-5}\)
& \(98.52\pm0.45\)
& \(9.62\pm0.10\)
& \(12.5830\pm0.9155\) \\
&
& iBDCA
& \((\mathbf{-1.9427}\pm0.2519)\times10^{-3}\)
& \(\mathbf{91.80}\pm0.78\)
& \((\mathbf{2.7450}\pm1.3601)\times10^{-5}\)
& \(\mathbf{99.10}\pm0.41\)
& \(10.86\pm0.12\)
& \(14.3282\pm1.0141\) \\
\cline{2-9}
& \(2\)
& DCA
& \(0\)
& --
& \((2.1837\pm1.4014)\times10^{-5}\)
& \(99.28\pm0.35\)
& \(\mathbf{8.32}\pm0.06\)
& \(\mathbf{10.2419}\pm0.7217\) \\
&
& BDCA
& \((-3.6374\pm0.2978)\times10^{-4}\)
& \(89.19\pm0.83\)
& \((1.7596\pm1.2181)\times10^{-5}\)
& \(99.35\pm0.33\)
& \(10.21\pm0.08\)
& \(12.6528\pm0.8957\) \\
&
& iBDCA
& \((\mathbf{-3.0923}\pm0.1577)\times10^{-3}\)
& \(\mathbf{94.66}\pm0.73\)
& \((\mathbf{3.9239}\pm3.9897)\times10^{-6}\)
& \(\mathbf{99.71}\pm0.26\)
& \(11.93\pm0.19\)
& \(15.0729\pm0.9885\) \\
\hline\hline
\textbf{S\&P 500}
& \(1.5\)
& DCA
& \(0\)
& --
& \((5.4816\pm17.9187)\times10^{-7}\)
& \(99.98\pm0.05\)
& \(\mathbf{8.90}\pm0.06\)
& \(\mathbf{11.2023}\pm0.2666\) \\
&
& BDCA
& \((-2.7996\pm0.1597)\times10^{-4}\)
& \(93.09\pm1.18\)
& \((5.4705\pm17.8803)\times10^{-7}\)
& \(99.98\pm0.05\)
& \(11.18\pm0.13\)
& \(14.1878\pm0.4199\) \\
&
& iBDCA
& \((\mathbf{-4.2599}\pm0.0841)\times10^{-3}\)
& \(\mathbf{98.89}\pm0.42\)
& \((\mathbf{5.1799}\pm17.9437)\times10^{-7}\)
& \(\mathbf{99.99}\pm0.04\)
& \(14.40\pm0.18\)
& \(19.0239\pm0.5020\) \\
\cline{2-9}
& \(1.75\)
& DCA
& \(0\)
& --
& \(\mathbf{0}\pm0\)
& \(\mathbf{100.00}\pm0.00\)
& \(\mathbf{9.03}\pm0.04\)
& \(\mathbf{11.1232}\pm0.6823\) \\
&
& BDCA
& \((-2.9076\pm0.1373)\times10^{-4}\)
& \(93.47\pm0.84\)
& \(\mathbf{0}\pm0\)
& \(\mathbf{100.00}\pm0.00\)
& \(11.31\pm0.10\)
& \(14.0562\pm0.8911\) \\
&
& iBDCA
& \((\mathbf{-5.6599}\pm0.0921)\times10^{-3}\)
& \(\mathbf{99.79}\pm0.19\)
& \((5.8746\pm19.4277)\times10^{-14}\)
& \(\mathbf{100.00}\pm0.00\)
& \(15.58\pm0.17\)
& \(20.2118\pm1.2734\) \\
\cline{2-9}
& \(2\)
& DCA
& \(0\)
& --
& \(\mathbf{0}\pm0\)
& \(\mathbf{100.00}\pm0.00\)
& \(\mathbf{9.04}\pm0.06\)
& \(\mathbf{10.8654}\pm0.7081\) \\
&
& BDCA
& \((-2.8604\pm0.1452)\times10^{-4}\)
& \(93.18\pm0.90\)
& \(\mathbf{0}\pm0\)
& \(\mathbf{100.00}\pm0.00\)
& \(11.16\pm0.14\)
& \(13.4934\pm0.9111\) \\
&
& iBDCA
& \((\mathbf{-6.5920}\pm0.1048)\times10^{-3}\)
& \(\mathbf{99.98}\pm0.05\)
& \((1.0893\pm2.8465)\times10^{-15}\)
& \(\mathbf{100.00}\pm0.00\)
& \(16.44\pm0.20\)
& \(20.7334\pm1.3499\) \\
\hline
\end{tabular}%
}

\caption{Robustness with respect to the VaR threshold \(\tau\) under
Scheme~2.}
\label{tab:tau_robustness}
\end{table}

\begin{remark}
The robustness results show that the relative behavior of the three
methods is stable across the tested VaR thresholds. For both datasets,
iBDCA attains the lowest penalized-objective values and the highest win
rates for all values of \(\tau\), while DCA remains the least expensive
method in terms of iteration counts and solution times.

For NASDAQ, tightening the VaR threshold has a visible effect on
scenario feasibility. At \(\tau=0.015\), the feasible rates of DCA,
BDCA, and iBDCA are \(92.70\%\), \(93.47\%\), and \(95.39\%\),
respectively. These rates increase to \(98.30\%\), \(98.52\%\), and
\(99.10\%\) at \(\tau=0.0175\), and to \(99.28\%\), \(99.35\%\), and
\(99.71\%\) at \(\tau=0.02\). The corresponding aggregate VaR
violations also decrease as \(\tau\) is relaxed, with iBDCA producing
the smallest mean violation at each tested threshold.

For the S\&P 500 dataset, all three methods remain essentially
scenario feasible throughout the tested range of \(\tau\), and the
aggregate VaR violations are zero or at numerical precision except for
very small values at the most restrictive threshold. Hence, the
advantage of iBDCA in attained penalized-objective value is observed
consistently across the tested risk thresholds without deterioration in
scenario-based VaR feasibility. Overall, the results indicate that the
solution-quality advantage of iBDCA is robust with respect to the VaR
threshold, although it comes at a higher computational cost.
\end{remark}

\subsection{Out-of-Sample Portfolio Performance}
\label{subsec:portfolio_performance}

In the out-of-sample backtest, each optimization method maintains its
own previous holding, so its wealth and transaction costs reflect its
own sequence of portfolio decisions. The EW strategy is rebalanced to
equal weights at every rolling period, whereas the BH strategy starts
from equal weights and is allowed to drift without subsequent
rebalancing. Transaction costs are charged whenever rebalancing is
performed.

Table~\ref{tab:portfolio_results} reports final wealth, the annualized
Sharpe ratio, maximum drawdown (MDD), empirical exceedance rate (ER),
and out-of-sample VaR residual. Wealth, Sharpe ratio, and MDD are
computed from net returns after transaction costs, whereas ER and the
VaR residual are computed from gross portfolio losses. For replication
\(r\), we define
\[
\operatorname{ER}^{(r)}
:=
\frac{1}{N_{\mathrm{OS}}}
\sum_{\ell=1}^{N_{\mathrm{OS}}}
\mathbf{1}
\left\{
L^{(r,\ell)}>\tau
\right\}
\; \text{and}\;
\operatorname{VR}_{\mathrm{OS}}^{(r)}
:=
\left(
\widehat{\operatorname{VaR}}_{\alpha}
\bigl(L^{(r)}\bigr)-\tau
\right)_+,
\]
where
\(L^{(r)}=(L^{(r,1)},\ldots,L^{(r,N_{\mathrm{OS}})})\) and
\(\widehat{\operatorname{VaR}}_{\alpha}(L^{(r)})\) is its empirical
\(\alpha\)-quantile. For DCA, BDCA, and iBDCA, the reported values are
the mean and standard deviation over the \(N_{\mathrm{run}}\)
replications, and boldface indicates the best mean among the three
methods within each initialization scheme.
\begin{table}[H]
\centering
\footnotesize
\setlength{\tabcolsep}{3pt}
\resizebox{\linewidth}{!}{%
\begin{tabular}{cclccccc}
\hline
Dataset
& Scheme
& Method
& \shortstack{Wealth\\\(\uparrow\)}
& \shortstack{Sharpe\\\(\uparrow\)}
& \shortstack{MDD\\\(\downarrow\)}
& \shortstack{ER\\\(\downarrow\)}
& \shortstack{OS VaR\\residual \(\downarrow\)} \\
\hline
\textbf{NASDAQ}
& -- & EW
& \(1.8811\)
& \(1.3063\)
& \(0.2006\)
& \(0.0347\)
& \(0\) \\
& -- & BH
& \(1.9741\)
& \(1.2550\)
& \(0.2156\)
& \(0.0454\)
& \(0\) \\
\hline
& 1 & DCA
& \(3.1441\pm0.2512\)
& \(1.4971\pm0.0939\)
& \(0.2427\pm0.0151\)
& \(0.0955\pm0.0047\)
& \(\mathbf{6.1856\times10^{-3}}\pm6.9270\times10^{-4}\) \\
& 1 & BDCA
& \(3.1322\pm0.2685\)
& \(1.4900\pm0.1011\)
& \(0.2426\pm0.0176\)
& \(\mathbf{0.0942}\pm0.0045\)
& \(6.2449\times10^{-3}\pm7.7550\times10^{-4}\) \\
& 1 & iBDCA
& \(\mathbf{3.1535}\pm0.2599\)
& \(\mathbf{1.4985}\pm0.0934\)
& \(\mathbf{0.2414}\pm0.0171\)
& \(0.0949\pm0.0044\)
& \(6.2664\times10^{-3}\pm6.4470\times10^{-4}\) \\
\hline
& 2 & DCA
& \(\mathbf{3.1502}\pm0.2516\)
& \(\mathbf{1.5004}\pm0.0914\)
& \(\mathbf{0.2429}\pm0.0151\)
& \(\mathbf{0.0951}\pm0.0042\)
& \(\mathbf{6.2942\times10^{-3}}\pm5.9890\times10^{-4}\) \\
& 2 & BDCA
& \(3.1415\pm0.2319\)
& \(1.4970\pm0.0852\)
& \(0.2433\pm0.0154\)
& \(0.0954\pm0.0049\)
& \(6.3574\times10^{-3}\pm7.0140\times10^{-4}\) \\
& 2 & iBDCA
& \(3.1357\pm0.2481\)
& \(1.4914\pm0.0927\)
& \(0.2467\pm0.0158\)
& \(0.0957\pm0.0042\)
& \(6.3786\times10^{-3}\pm6.6360\times10^{-4}\) \\
\hline\hline
\textbf{S\&P 500}
& -- & EW
& \(1.8338\)
& \(1.8021\)
& \(0.1400\)
& \(0.0067\)
& \(0\) \\
& -- & BH
& \(1.8344\)
& \(1.7423\)
& \(0.1519\)
& \(0.0067\)
& \(0\) \\
\hline
& 1 & DCA
& \(\mathbf{2.2957}\pm0.1062\)
& \(\mathbf{1.3499}\pm0.0660\)
& \(\mathbf{0.2321}\pm0.0107\)
& \(0.0531\pm0.0027\)
& \(8.0639\times10^{-4}\pm5.0903\times10^{-4}\) \\
& 1 & BDCA
& \(2.2512\pm0.1278\)
& \(1.3217\pm0.0811\)
& \(0.2331\pm0.0112\)
& \(0.0528\pm0.0027\)
& \(7.4704\times10^{-4}\pm5.6140\times10^{-4}\) \\
& 1 & iBDCA
& \(2.1958\pm0.1441\)
& \(1.2801\pm0.0954\)
& \(0.2380\pm0.0166\)
& \(\mathbf{0.0526}\pm0.0032\)
& \(\mathbf{6.8527\times10^{-4}}\pm5.6762\times10^{-4}\) \\
\hline
& 2 & DCA
& \(\mathbf{2.3001}\pm0.1184\)
& \(\mathbf{1.3533}\pm0.0769\)
& \(\mathbf{0.2321}\pm0.0112\)
& \(\mathbf{0.0532}\pm0.0029\)
& \(8.5719\times10^{-4}\pm5.7624\times10^{-4}\) \\
& 2 & BDCA
& \(2.2586\pm0.1222\)
& \(1.3276\pm0.0810\)
& \(0.2337\pm0.0123\)
& \(0.0542\pm0.0030\)
& \(8.1426\times10^{-4}\pm5.0153\times10^{-4}\) \\
& 2 & iBDCA
& \(2.2006\pm0.1569\)
& \(1.2832\pm0.1045\)
& \(0.2387\pm0.0152\)
& \(0.0537\pm0.0030\)
& \(\mathbf{8.0397\times10^{-4}}\pm5.8977\times10^{-4}\) \\
\hline
\end{tabular}%
}
\caption{Out-of-sample portfolio performance.}
\label{tab:portfolio_results}
\end{table}
Figure~\ref{fig:wealth_trajectories} shows the median wealth
trajectories and interquartile bands for DCA, BDCA, and iBDCA. The
left and right panels correspond to Schemes~1 and~2, respectively.
Since EW and BH are deterministic for each dataset, they are
represented by single wealth trajectories.

\begin{figure}[H]
\centering
\begin{subfigure}{0.48\textwidth}
\centering
\includegraphics[width=\linewidth]
{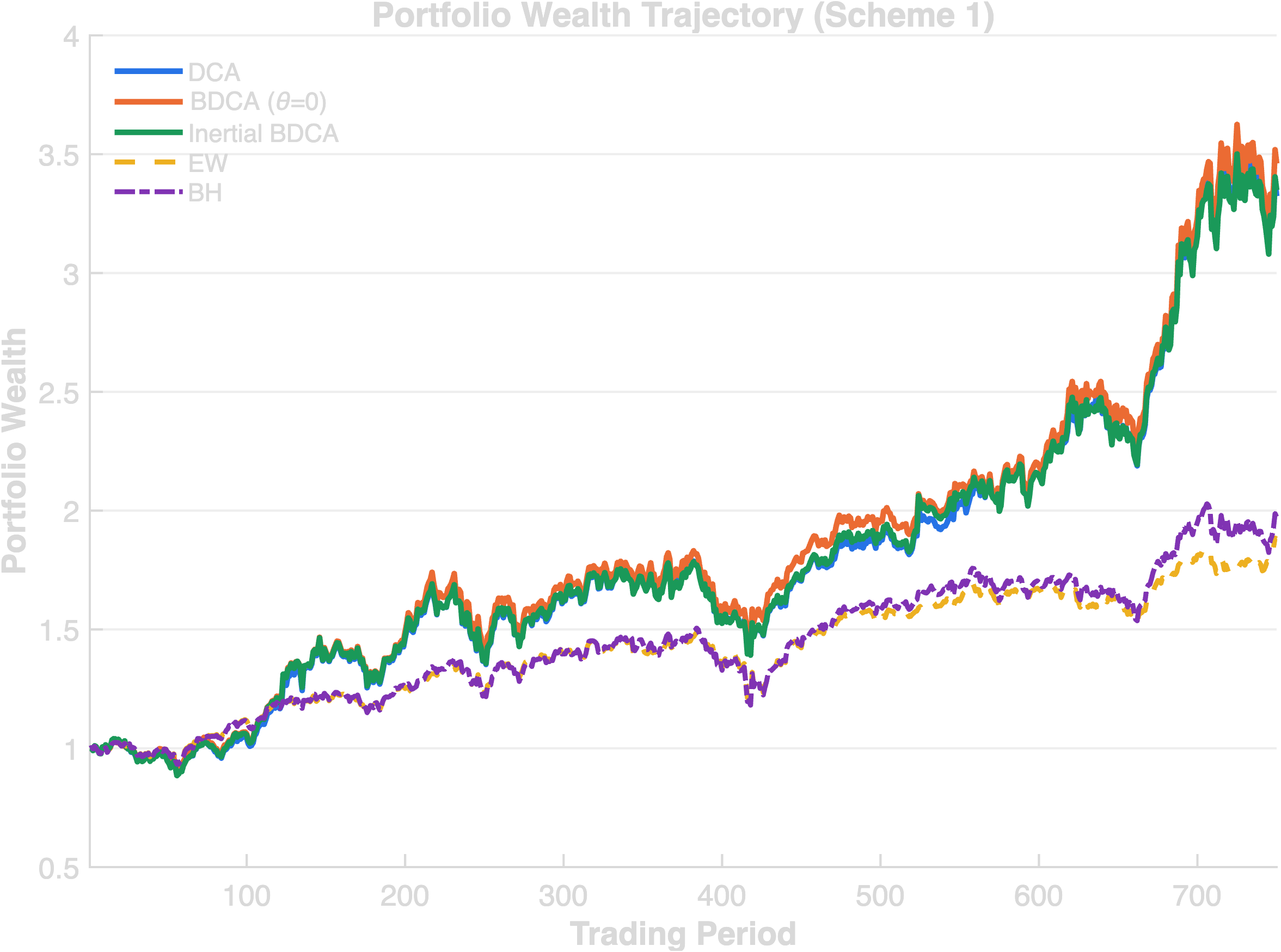}
\caption{NASDAQ scheme 1}
\end{subfigure}
\hfill
\begin{subfigure}{0.48\textwidth}
\centering
\includegraphics[width=\linewidth]
{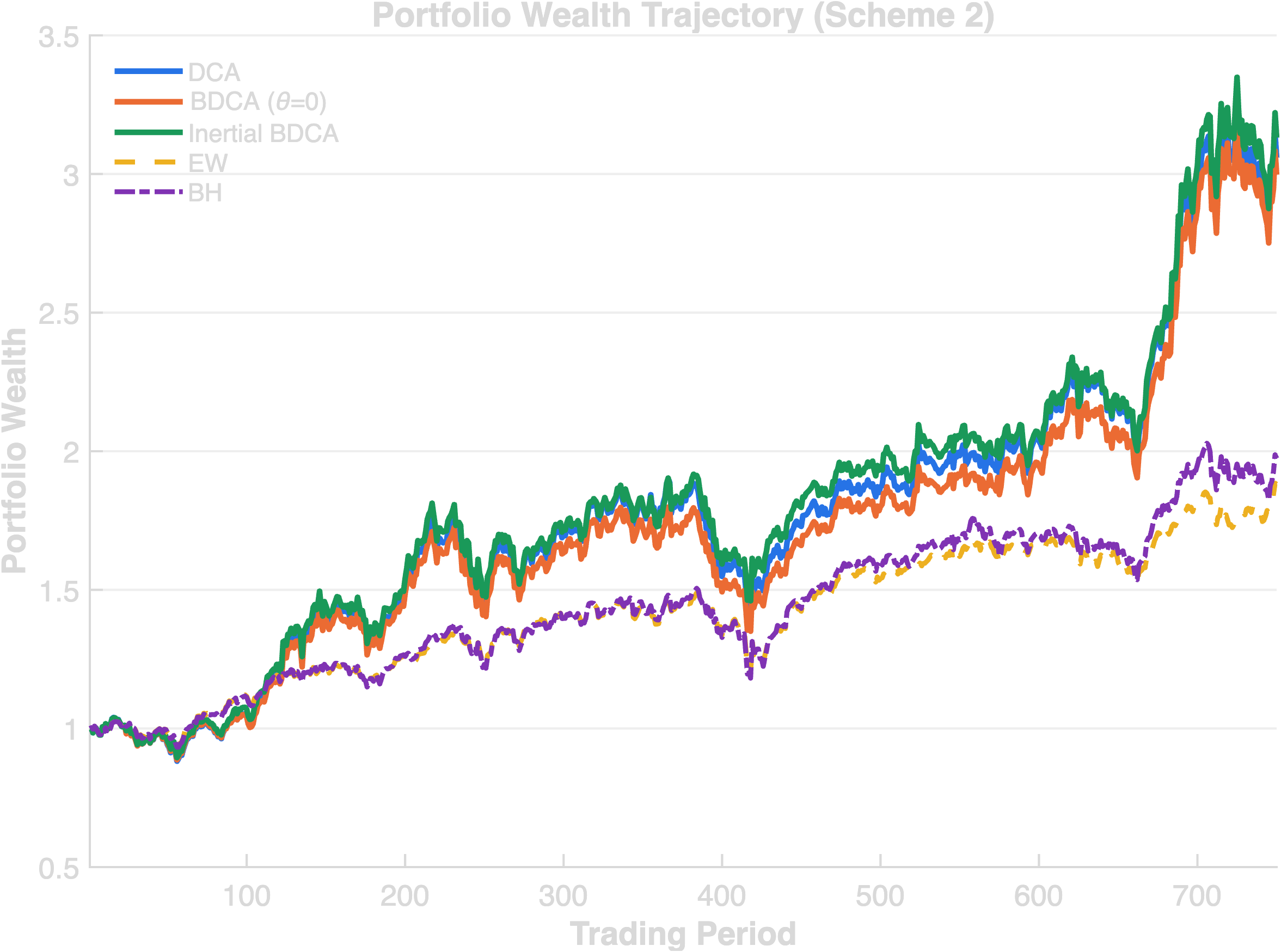}
\caption{NASDAQ scheme 2}
\end{subfigure}

\medskip

\begin{subfigure}{0.48\textwidth}
\centering
\includegraphics[width=\linewidth]
{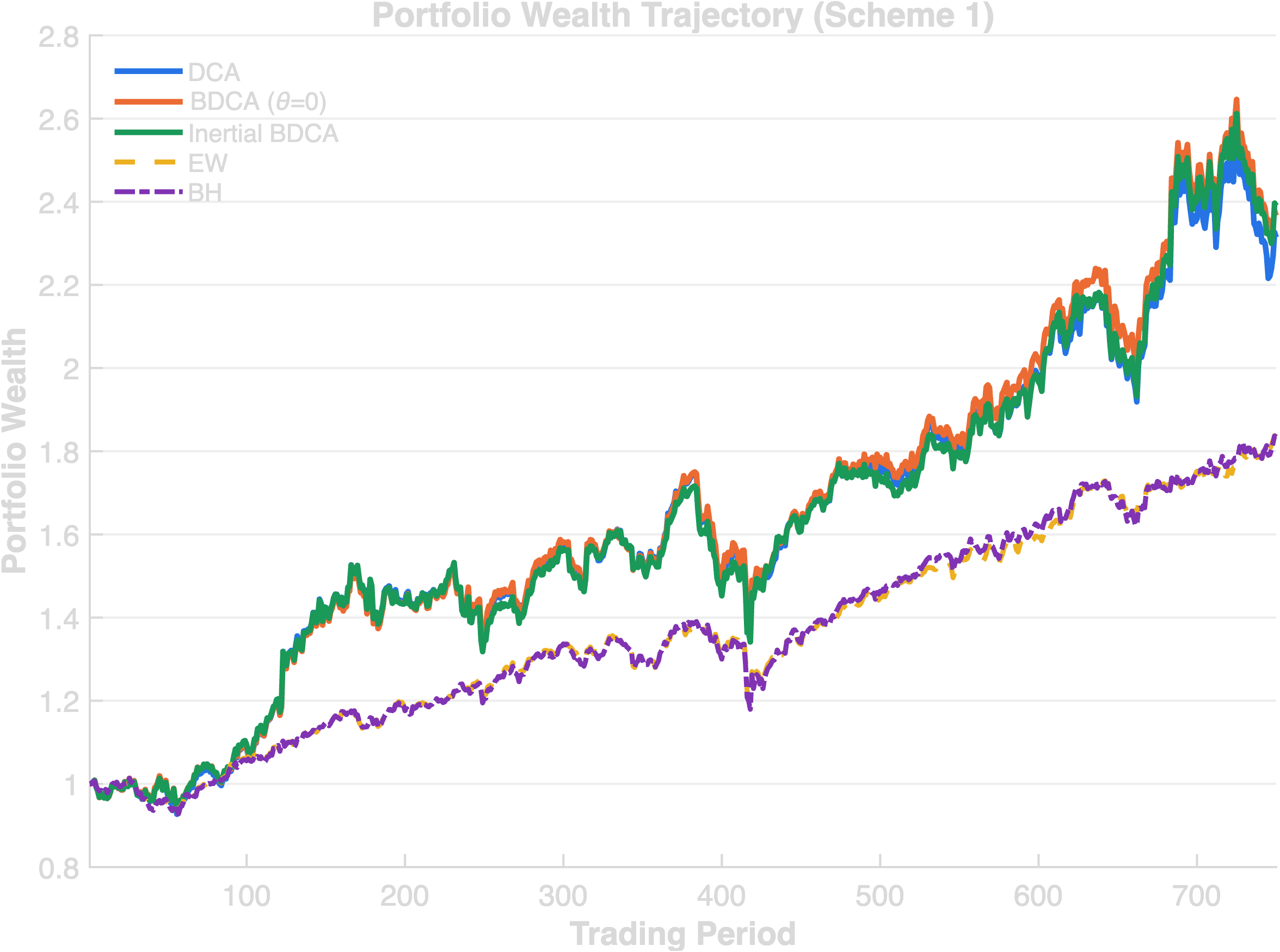}
\caption{S\&P 500 scheme 1}
\end{subfigure}
\hfill
\begin{subfigure}{0.48\textwidth}
\centering
\includegraphics[width=\linewidth]
{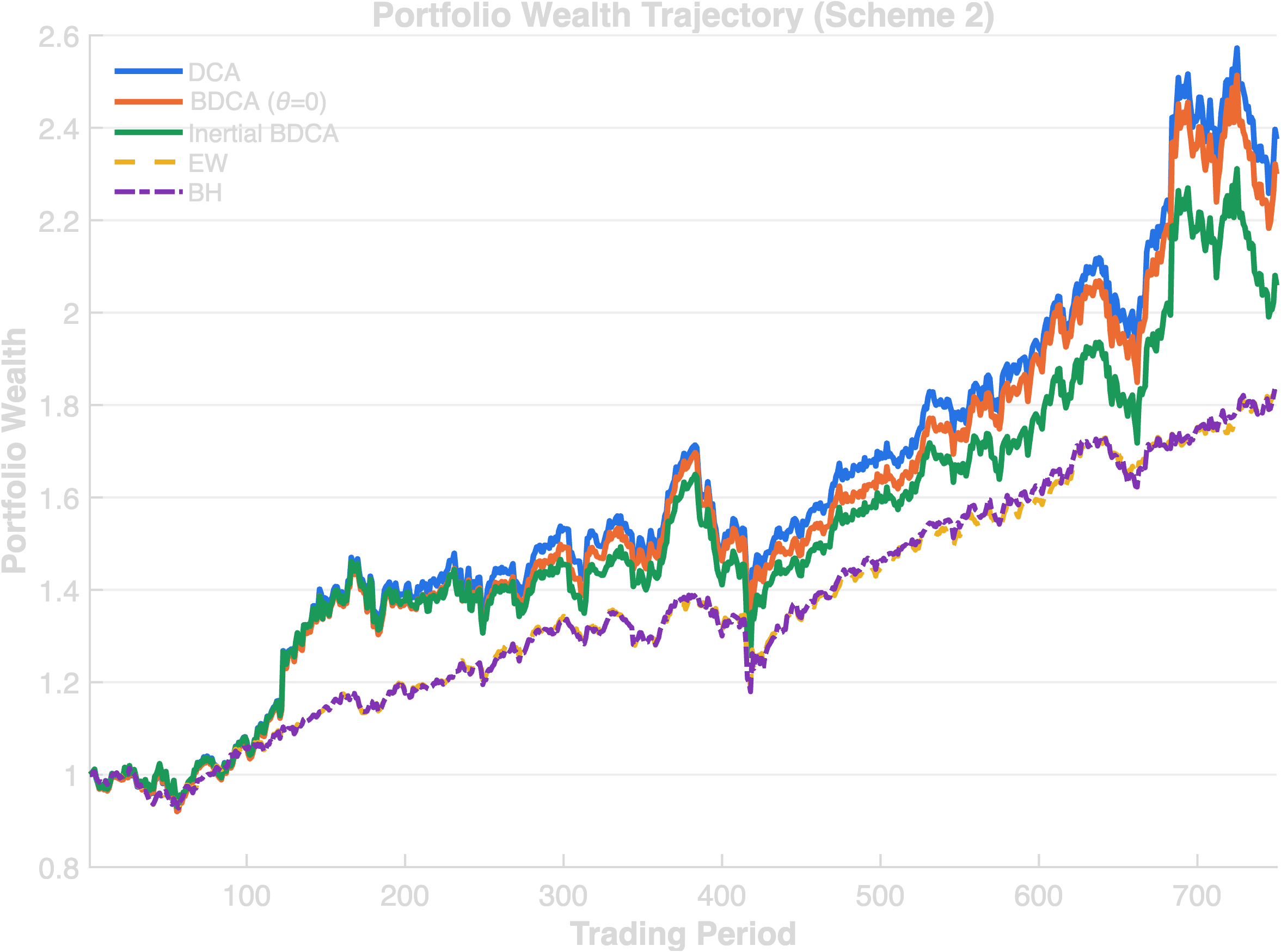}
\caption{S\&P 500 scheme 2}
\end{subfigure}

\caption{Out-of-sample wealth trajectories.}
\label{fig:wealth_trajectories}
\end{figure}

\begin{remark}
The out-of-sample results reveal different return--risk patterns across
the two datasets. For NASDAQ, the optimized strategies attain higher
final wealth and Sharpe ratios than EW and BH, but also exhibit larger
drawdowns. Their exceedance rates are around \(9.5\%\), well above the
nominal level \(1-\alpha=5\%\), and their VaR residuals remain
positive. Thus, the higher realized performance is accompanied by
weaker out-of-sample risk control. Among the three optimization
methods, the relative performance also depends on the initialization
scheme: iBDCA performs best in several criteria under Scheme~1,
whereas DCA performs best across all reported criteria under Scheme~2.

For the S\&P 500, the optimized strategies again attain higher final
wealth, whereas EW and BH provide higher Sharpe ratios and smaller
drawdowns. In contrast to NASDAQ, the exceedance rates of DCA, BDCA,
and iBDCA remain close to the nominal \(5\%\) level, and the
corresponding VaR residuals are relatively small. Although iBDCA
attains lower penalized-objective values in both the matched solver
and robustness experiments, this penalized-objective advantage does
not translate consistently into superior out-of-sample portfolio
performance. Overall, no single method among DCA, BDCA, and iBDCA
dominates the others across both datasets, initialization schemes, and
reported out-of-sample criteria.
\end{remark}

\section{Conclusion}
\label{sec6}

This paper considered a multi-period portfolio optimization problem
with VaR constraints under a finite-scenario approximation,
transaction costs, and diversification 
regularization. The resulting problem is nonsmooth and nonconvex.
Using an exact finite-scenario VaR--CVaR representation, we formulated
it as a penalized DC problem over the convex portfolio set.

We proposed a projected inertial BDCA for solving the penalized
problem. The method combines inertial extrapolation, projection, an
objective safeguard, a strongly convex DCA subproblem, and a
backtracking line search. We proved that the algorithm is well defined,
generates a decreasing objective sequence, and has only critical
accumulation points for the penalized DC problem. Under a local
no-ties condition at an accumulation point, we further established
whole-sequence convergence with a local \(R\)-linear rate.

The numerical investigation consisted of matched solver experiments
and out-of-sample portfolio backtests. On the matched instances, iBDCA
frequently attained lower penalized objective values than DCA and
standard BDCA, although generally at the cost of additional iterations
and solution time. The portfolio backtests additionally included
equal-weight and buy-and-hold benchmarks and evaluated final wealth,
Sharpe ratio, maximum drawdown, and empirical VaR measures. No single method performed best across all evaluation criteria. In
particular, a lower penalized objective value did not always lead to
higher realized wealth or better out-of-sample risk measures. The
results therefore illustrate the trade-offs among return, risk,
transaction costs, and empirical VaR control.
\begin{table}[H]
\centering
\footnotesize
\setlength{\tabcolsep}{1.5pt}
\renewcommand{\arraystretch}{0.95}

\resizebox{\linewidth}{!}{%
\begin{tabular}{cccccccccc|cccccccccc}
\hline
\multicolumn{10}{c|}{\textbf{NASDAQ dataset}}
&
\multicolumn{10}{c}{\textbf{S\&P 500 dataset}} \\
\hline
AAPL & ADBE & ADI & ADP & AMD & AMGN & AMZN & AVGO & BKNG & CMCSA
&
AAPL & MSFT & IBM & JPM & BAC & GS & JNJ & MRK & ABBV & AMZN
\\
COST & CSCO & GILD & GOOGL & HON & INTU & ISRG & LRCX & MDLZ & META
&
HD & MCD & PG & KO & WMT & CAT & GE & RTX & XOM & CVX
\\
MSFT & NFLX & NVDA & ORCL & PANW & PEP & PYPL & QCOM & REGN & SBUX
&
COP & GOOGL & META & VZ & NEE & DUK & LIN & APD & AMT & PLD
\\
\hline
\end{tabular}%
}

\caption{Ticker symbols of the assets used in the two datasets.}
\label{tab:asset_lists}
\end{table}

\end{document}